\documentclass{amsart}

\usepackage[T1]{fontenc}
\usepackage[utf8]{inputenc}
\usepackage{babel}

\usepackage{color}

\definecolor{blue}{rgb}{0,0,0.8}

\definecolor{darkblue}{rgb}{0,0,0.4}
\definecolor{darkred}{rgb}{0.4,0,0}
\definecolor{darkgreen}{rgb}{0,0.4,0}

\usepackage{graphicx}
\usepackage{amsmath}
\usepackage{amsthm}
\usepackage{amssymb}
\usepackage{amsfonts}
\usepackage{latexsym}

\newcommand{\con}{{\rm con}\,}

\newcommand{\PP}{{\mathcal P}}

\newcommand{\ze}{\zeta}
\newcommand{\zs}{\zeta(s)}
\newcommand{\DD}{{\mathcal D}}
\newcommand{\GG}{{\mathcal G}}

\newcommand{\G}{{\mathcal G}}
\newcommand{\A}{{\mathcal A}}
\newcommand{\AAA}{{\mathcal A}}

\newcommand{\C}{{\mathcal C}}

\newcommand{\Cn}{{\mathcal C}_0}
\newcommand{\Co}{{\mathcal C}_1}
\newcommand{\HH}{{\mathbb H}}

\newcommand{\rks}{{\rho}_k^*}
\newcommand{\zz}{\mbox{\boldmath$z$}}
\newcommand{\ww}{\mbox{\boldmath$w$}}

\newcommand{\Z}{{\mathcal Z}}

\newcommand{\RR}{{\mathbb R}}
\newcommand{\CC}{{\mathbb C}}

\newcommand{\NN}{{\mathbb N}}

\newcommand{\N}{{\mathcal N}}
\newcommand{\R}{{\mathcal R}}

\newcommand{\MM}{{\mathcal M}}

\newcommand{\Ds}{{D(\eta)^*}}

\newcommand{\D}{{\Delta}}

\newcommand{\al}{{\alpha}}
\newcommand{\ve}{{\varepsilon}}
\newcommand{\de}{\delta}
\newcommand{\si}{\sigma}
\newcommand{\ff}{\varphi}

\newcommand{\vp}{{\varepsilon}'}
\newcommand{\ot}{\widetilde{\omega}}
\newcommand{\be}{\beta}
\newcommand{\ga}{\gamma}
\newcommand{\om}{\omega}
\newcommand{\oet}{\omega_\eta}

\newcommand{\zp}{\zeta_{\mathcal P}}

\newcommand{\Dp}{\Delta_{\mathcal P}}

\renewcommand{\refname}{}

\newtheorem{theorem}{Theorem}
\newtheorem{corollary}{Corollary}
\newtheorem{lemma}{Lemma}
\newtheorem{proposition}{Proposition}
\newtheorem{definition}{Definition}

\theoremstyle{definition}

\newtheorem{remark}{Remark}

\reversemarginpar 

\begin{document}

\title[The Ingham question on the order of the PNT in the Beurling context]{The method of Pintz for the Ingham question about the connection of distribution of $\zeta$-zeroes and order of the error in the PNT in the Beurling context}

\author{Szil\' ard Gy. R\' ev\' esz}

\maketitle

\begin{abstract} We prove two results, generalizing long existing knowledge regarding the classical case of the Riemann zeta function and some of its generalizations. These are concerned with the question of Ingham who asked for optimal and explicit order estimates for the error term $\Delta(x):=\psi(x)-x$, given any zero-free region $\DD(\eta):=\{s=\si+it\in\CC~:~ \si:=\Re s \ge 1-\eta(t)\}$. In the classical case essentially sharp results are due to some 40 years old work of Pintz.

Here we consider a given system of Beurling primes $\PP$, the generated arithmetical semigroup $\GG$, the corresponding integer counting function $\N(x)$, and the respective error term $\Dp(x):=\psi_{\PP}(x)-x$ in the PNT of Beurling, where $\psi_{\PP}(x)$ is the Beurling analog of $\psi(x)$. First we prove that if the Beurling zeta function $\zp$ does not vanish in $\DD(\eta)$, then the extension of Pintz' result holds:
$|\Dp(x)| \le x\exp((1+\ve)\omega_\eta(\log x))~(x>x_0(\ve))$,
where $\omega_\eta(y):={\mathcal{ L}}(\eta)(y)$ is the naturally occurring conjugate function--essentially the Legendre transform--of $\eta(t)$, introduced into the field by Ingham.
In the second part we prove a converse: if $\zp$ has an infinitude of zeroes in the given domain, then analogously to the classical case, $|\Dp(x)| \ge x\exp((1-\ve)\omega_\eta(\log x))$
holds "infinitely often". This also shows that both main results are sharp apart from the arbitrarily small $\ve>0$.


The  classical results of Pintz used many facts about the Riemann zeta function. Recently we worked out a number of analogous results--including some construction of quasi-optimal integration paths, a Riemann-von Mangoldt type formula, a Carlson-type density theorem and a Turán type local density theorem--for the Beurling context, too. These, together with Turán's power sum theory, all play some indispensable role in deriving the main results of the paper.
\end{abstract}



\medskip
{\bf Author information.} Alfréd Rényi Institute of Mathematics\\
Reáltanoda utca 13-15, 1053 Budapest, Hungary \\
{\tt revesz.szilard@renyi.hu}


\section{Introduction}
In Beurling's theory of generalized integers and primes, $\G$ is a unitary, commutative semigroup, with a countable set $\PP$ of indecomposable generators, called the \emph{primes} of $\G$, which freely generate\footnote{In exact terms this means that any element $g\in \G$ can be uniquely written (up to order of
terms) in the form $g=p_1^{k_1}\cdot \dots \cdot p_m^{k_m}$: two (essentially) different such expressions are necessarily different as elements of $\G$, while each element has
its (essentially) own unique prime decomposition.} the whole of $\G$.

Moreover, there is a \emph{norm} $|\cdot|~: \G\to \RR_{+}$ such that
the following hold. First, the image of $\G$, $|\G|\subset \RR_{+}$
is \emph{locally finite} 
in the sense that any finite interval of $\RR_{+}$ can contain the norm of only a finite number of elements of $\G$; thus the
function
\begin{equation}\label{Ndef}
{\N}(x):=\# \{g\in \G~:~ |g| \leq x\}
\end{equation}
exists as a finite, nondecreasing, right continuous, nonnegative integer valued
function on $\RR_{+}$. Second, the norm is multiplicative, i.e. $|g\cdot h| = |g| \cdot
|h|$; it follows that for the unit element $e$ of $\G$ $|e|=1$, and
that all other elements $g \in \G$ have norms strictly larger than 1.


In this paper we will assume throughout\footnote{We indeed assume Axiom A even wherever it is not stated explicitly. Without a meromorphic continuation of $\zeta(s)$ all our analysis--in particular referring to zeroes of $\zeta(s)$--would be void, thus without Axiom A we should have assumed other conditions to ensure meromorphic continuity. Moreover, it is well-known that in case Axiom A fails to hold, a meromorphic continuation, even if it may exist, but can behave rather wildly. E.g. if the primes are well-behaved, then $\zeta(s)$ must have infinite order of growth \cite{H-7}. A construction of very well-behaved primes with widely oscillating integer distribution (and thus with $\zeta(s)$ having infinite order left to $\Re s=1$) is worked out in \cite{BrouckeDebruyneVindas}. Moreover, according to an effective Ikehara theorem \cite{Aramaki}, the following is true: if $\zs(s-1)$ has an analytic continuation to $\Re s >\theta$ and admits a polynomial order bound $|\zs|\ll (1+|s|)^M$ ($|\Im s|\ge 1$), then $\N(x)=\kappa x + O(x^q)$, with some $q<1$ (irrespective of any condition on the behavior of the prime counting function). In other words, \emph{not having Axiom A} with some $q<1$ necessarily entails infinite order of growth. This of course makes all the technical lemmas and estimates--which are very crucially used here--impossible to get. This should well explain why we restrict ourselves to this condition holding everywhere in this work.} the so-called \emph{"Axiom A"} of Knopfmacher\footnote{In the terminology of Hilberdink \cite{H-5}, such conditions are termed as being "well-behaved", e.g. the integers here.}, see pages 73-79 of \cite{Knopf}.

\begin{definition} ${\N}$ satisfies \emph{Axiom A} -- more precisely, Axiom
$A(\kappa,\theta)$ with the constants $\kappa>0$ and
$0<\theta<1$ -- if we have\footnote{The usual formulation uses the more natural version $\R(x):= \N(x)-\kappa x$. However, our version is more convenient with respect to the initial values at 1, as we here have $\R(1-0)=0$. All respective integrals of the form $\int_1$ will be understood as integrals from $1-0$, and thus we can avoid considering endpoint values in the respective partial integration formulae. Alternatively, we could have taken also $\N(x)$ left continuous: also with this convention we would have $\R(1)=0$.} for the remainder term $\R(x):= \N(x)-\kappa (x-1)$ the estimate
\begin{align}\label{Athetacondi}
\left| \R(x) \right|  \leq A x^{\theta} \quad (\kappa, A > 0, ~ 0<\theta<1 ~ \textrm{constants}, ~ x \geq 1 ~ \textrm{arbitrary}).
\end{align}
\end{definition}
The Beurling zeta function is defined as the Mellin transform of $\N(x)$, i.e.
\begin{equation}\label{zetadef}
\ze(s):=\ze_{\G}(s):=\MM(\N)(s):=\int_1^{\infty} x^{-s} d\N(x) = \sum_{g\in\G} \frac{1}{|g|^s}.
\end{equation}
If only $\N(x)=O(x^C)$, the series converges absolutely and locally uniformly in the halfplane $\Re s> C$, moreover, its terms can be rearranged to provide the \emph{Euler product formula}
\begin{equation}\label{Euler}
\ze_{\G}(s)=\prod_{p\in\PP} \left(\frac{1}{1-|p|^{-s}}\right).
\end{equation}
In particular, if $\N(x)=O(x^{1+\ve})$ for all $\ve>0$, then $\zeta_\GG$ is absolutely convergent in $\Re s>1$, it cannot vanish there--as is clear from \eqref{Euler}--moreover, $|\zs| \ge 1/\zeta(\si)$ ($\si:=\Re s$). Furthermore, under Axiom A it admits a meromorphic, essentially analytic continuation
$\kappa\frac{1}{s-1}+\int_1^{\infty} x^{-s} d\R(x)$ up to $\Re s >\theta$
with only one simple pole at 1. For an analysis of the finer behavior of the number of primes $\pi_{\PP}(x):=\sum_{p\in \PP;~|p|\le x} 1$--as in the classical case of $\GG=\NN$--the location of the zeroes of $\zs$ in the "critical strip" $\theta<\Re s\le 1$ is decisive, as we will see.

For a Beurling system 
the generalized von Mangoldt function is
\begin{equation}\label{vonMangoldtLambda}
\Lambda (g):=\Lambda_{\G}(g):=\begin{cases} \log|p| \quad \textrm{if}\quad g=p^k,
~ k\in\NN ~~\textrm{with some prime}~~ p\in\G\\
0 \quad \textrm{if}\quad g\in\G ~~\textrm{is not a prime power in} ~~\G
\end{cases}.
\end{equation}
These appear also as the coefficients of the logarithmic derivative of the Beurling zeta function:
\begin{equation}\label{zetalogder}
-\frac{\zeta'}{\zeta}(s) = \sum_{g\in \G} \frac{\Lambda(g)}{|g|^s}.
\end{equation}

The Beurling theory of generalized primes investigates the summatory function
\begin{equation}\label{psidef}
\psi(x):=\psi_{\G}(x):=\sum_{g\in \G,~|g|\leq x} \Lambda (g).
\end{equation}
The asymptotic relation $\psi(x)\thicksim x$ is equivalent to say that $\pi(x):=\sum_{p\in\PP,\, |p|\le x}1$ satisfies $\pi(x) \thicksim {\rm li}(x):=\int_2^x du/\log u$ or $\pi(x) \thicksim x/\log x$, and is thus termed as \emph{the Prime Number Theorem} (PNT). Equivalently, we can also formulate this by use of the "error term in the prime number formula", for which the standard notation is
\begin{equation}\label{Deltadef}
\Delta(x):=\Delta_{\G}(x):=\psi(x)-x.
\end{equation}
Then PNT is thus the statement that $\Delta(x)=o(x)$. The so-called "Chebyshev bounds" $x \ll \psi(x) \ll x$ mean that there are constants $0<c<C<\infty$ with $cx \le \psi(x) \le Cx$. Much study was devoted to describe, what conditions are necessary resp. sufficient for the PNT or the Chebyshev bounds to hold in the generality of Beurling systems \cite{Beur, DebruyneVindas-PNT, DSV, DZ-17, K-98, DZ-13-3, DZ-13-2, Vindas12, Vindas13, Zhang15-IJM, Zhang15-MM, Zhang93}. Another much studied question, which simply does not arise in the classical case, is density and asymptotics of $\N(x)$ \cite{Diamond-77, K-17, DebruyneVindas}. Still another direction, going back to Beurling himself, is the study of analogous questions in case the assumption of Axiom A is weakened to e.g. an asymptotic condition on $\N(x)$ with a product of $x$ and a sum of powers of $\log x$, or a sum of powers of $\log x$ perturbed by almost periodic polynomials in $\log x$, or $N(x)-cx$ being periodic, see \cite{Beur, H-12, RevB, Zhang93}.

There are other studies related to the Beurling PNT in the literature. In particular, some rough (as compared to our knowledge in the prime number case) estimates and equivalences were worked out in the analysis of the connection between $\zeta$-zero distribution and the behavior of the error term $\Delta(x)$. One of the deep results in this direction is the extension of the classical oscillation result $\pi(x)-{\rm li}(x) =\Omega_{\pm}(\sqrt{x}\log\log\log x/\log x)$ of Littlewood \cite{Littlewood-CR} to the Beurling context \cite{K-99}. Further, so-called $(\alpha,\beta)$ systems and $(\alpha,\beta,\gamma)$ systems were defined \cite{H-20}, with these parameters denoting the "best possible" exponents in estimating the error terms $\R(x), \Delta(x)$ and the summatory function $M_{\GG}(x)$ of the Beurling version of the M\"obius function $\mu_{\GG}(g)$; in particular, it is known that the two largest of these three parameters have to be at least 1/2 and must match \cite{H-5}, \cite{H-20}. Oscillation order of the generalized M\"obius summatory function and even more general arithmetical functions are also treated up to recent times \cite{H-10, DZ-17, DebruyneDiamondVindas, BrouckeVindas, PintzMP}.

The interest in the Beurling theory was greatly boosted by a construction of Diamond, Montgomery and Vorhauer \cite{DMV}. They basically showed that under Axiom A
the Riemann hypothesis may still fail; moreover, nothing better
than the most classical \cite{V} zero-free region and error term  of
\begin{equation}\label{classicalzerofree}
\zeta(s) \ne 0 \qquad \text{whenever}~~~ s=\sigma+it, ~~ \sigma >
1-\frac{c}{\log t},
\end{equation}
and
\begin{equation}\label{classicalerrorterm}
\psi(x)=x +O(x\exp(-c\sqrt{\log x}))
\end{equation}
follows from \eqref{Athetacondi} at least if $\theta>1/2$.

Therefore, e.g. Vinogradov estimates and many other technology cannot prevail in this generality,
and for Beurling zeta functions only a restricted variety of arguments can be implemented.

After the Diamond-Montgomery-Vorhauer paper \cite{DMV}, better and better examples were constructed for arithmetical semigroups with very "regular" and very "irregular" behavior on the zero- or prime distribution, see e.g. \cite{H-15}, \cite{BrouckeDebruyneVindas}, \cite{DMV}, \cite{H-5}, \cite{Zhang7}. For an analysis of these directions as well as for much more information the reader may consult the monograph \cite{DZ-16}.

In sum, in contrast with the classical natural number system, when it is generally believed that the Riemann Hypothesis holds true, in the generality of arithmetical semigroups many different scenarios occur. It is all the more natural to extend the questions, originally posed in the classical case by Littlewood \cite{Littlewood-JLMS} and of Ingham \cite{Ingham-AA}, what \emph{explicit, effective} conclusions can be drawn for the oscillation of the error term $\Delta(x)$ from the existence of a given $\zeta$-zero, or from the existence of an infinitude of $\zeta$-zeroes within a given domain, and, conversely, what can be the best order estimates assuming a given zero-free region. However, to date, only a few attempts were made to \emph{generally} describe the intimate correspondence between the location of the zeroes of the Beurling zeta function $\zeta$ on the one hand and order estimates or oscillation results for the remainder term $\Delta(x)$ on the other hand.

In the classical case of the Riemann zeta function the problem of Littlewood was first answered by Turán \cite{Turan1}, based on his celebrated power sum method  \cite{Turan}, \cite{SosTuran}. The Turán result then was sharpened in several steps \cite{Stas0}, \cite{Knapowski}, \cite{Pintz1}, \cite{RevAA}, and were extended to various more general contexts, in particular to the case of prime ideals of algebraic number fields, see \cite{Stas-Wiertelak-1, Stas1, Rev1, RevAA}. These effective results also furnished some localizations, where the large oscillations should occur, while related works \cite{PintzBasel, PintzNoordwijkerhout, Schlage-Puchta, PintzProcStekl} produced various versions where the sharpness of the estimate was a little bit sacrificed in exchange for sharper localizations, a trade-off rather characteristic in these results.

Let us describe in more detail the development around Ingham's problem, as it will be our main topic here for the Beurling case. We start with upper bounds following from localization of the set of zeroes of the Riemann $\zeta$ function. Denote by $\eta(t):[0,\infty)\to [0,1/2]$ a nonincreasing function and consider the domain
\begin{equation}\label{eq:etazerofree}
\DD(\eta):=\{ s=\sigma+it \in\CC~:~ \sigma>1-\eta(|t|),~ t \in \RR\}.
\end{equation}
Ingham--see the last line of page 62 in \cite{Ingham}--defined the function $\om_\eta$ formally\footnote{This can as well be expressed by means of the classical Legendre transform of a function:
$$
\min_{v:=\log t\ge 0} \eta(e^v) u + v =\min_{w:=\phi(v):=-\eta(e^v)} -w u+ \phi^{-1}(w)=-\sup_{w} wu-\phi^{-1}(w)=:-{\mathcal L}(\phi^{-1})(u).
$$
}
as
\begin{equation}\label{eq:omegaetadef}
\om_\eta(x):=\min_{t\ge 1} \eta(t)\log x + \log t=\min_{v:=\log t\ge 0} \eta(e^v) u + v,
\end{equation}
where here we also wrote in the last form the logarithmic variables $v:=\log t$ and $u:=\log x$. It will be called Ingham's function corresponding to $\eta$. With this, Ingham showed the following \cite{Ingham-AA}.
\begin{theorem}[Ingham]\label{th:Ingham} Assume that for the Riemann zeta function $\zs\ne 0$ in the domain $\DD(\eta)$ in \eqref{eq:etazerofree}, i.e., for $\Re s \ge 1-\eta(|t|)$, where $s=\si+it$ and $\eta:[0,\infty) \to [0,1/2]$ is a strictly decreasing continuously differentiable function with $1/\eta(t)=O(\log t)$ and $\eta'(t)$ tending to 0 with $t\to \infty$.

Then for arbitrary fixed $\ve>0$ we have $|\D(x)| =O(x\exp(-(1/2-\ve)\om_\eta(x)))$.
\end{theorem}
While in the opposite direction there were several results \cite{Stas2}, \cite{Turan2} it was Pintz who observed that Ingham's Theorem \ref{th:Ingham} involves a loss\footnote{In view of the necessary technicalities, we shall explain only later in the sequel why one should optimally expect $\approx x\exp(-\om_\eta(x))$ in Theorem \ref{th:Ingham}. In fact, for the integrated function $\Psi_1(x):=\int_1^x \psi(y)dy$ also Ingham obtained the "right order error", proving $\Psi_1(x)=\frac12 x^2+O(x^2\exp(-(1-\ve)\om_\eta(x)))$ in Theorem 21 on the bottom of page 62 of \cite{Ingham}. However, to deduce error estimates for $\psi(x)$ itself from formulae for $\Psi_1$ incurred some losses.} of a factor $1/2$ in the exponent, so that to have both ways sharp results one needs to improve upon that, too. Subsequently he indeed could prove the improved result in \cite{Pintz2}, while allowing more general functions than Ingham could.
\begin{theorem}[Pintz]\label{th:domainesti} Let $\eta: [0,\infty)\to [0,1/2]$ be a continuous nonincreasing function.

If $\zeta(s)\ne 0$ in $\DD(\eta)$, then for arbitrary $\ve>0$ we have $\Delta(x)=O(x\exp(-(1-\ve)\omega_\eta(x))).$
\end{theorem}
Next we come to results in the opposite direction, i.e., to lower estimation of the oscillation order of $\D(x)$ corresponding to assumption on penetration of zeroes into \eqref{eq:etazerofree}. The first effective results in this direction were also obtained by Turán \cite{Turan2} in the log power case and then Stas \cite{Stas2} for general curves. Later, these results were sharpened and extended to various cases, in particular to prime ideal distribution \cite{Stas-Wiertelak-2}, \cite{Knapowski}. For the natural number system, the optimal result was then achieved by Pintz in \cite{Pintz2}.
\begin{theorem}[Pintz]\label{th:domainosci}
Conversely, assume that there is an infinite sequence of zeroes of the Riemann $\zeta$ function in the domain \eqref{eq:etazerofree}, where $\eta(t)$ is a continuously differentiable strictly convex function.

Then we have for any $\ve >0$ the oscillation $\Delta(x) = \Omega(x\exp(-(1+\ve)\omega_\eta(x)))$.
\end{theorem}
This result then was extended to prime ideal distribution, too \cite{Rev2}.

The present work is part of a series. In \cite{Rev-MP} we proved a number of technical auxiliary results including the Riemann-von Mangoldt formula \eqref{eq:Riemann-vonMangoldt}. In \cite{Rev-D} we worked out three theorems on the distribution of zeroes of the Beurling zeta function in the critical strip, two of them proving to be crucial in our present work. First, we extended to the Beurling case a classic local density type estimate of Turán, see Lemma \ref{th:locdens}. Second, substantially strengthening a result of Kahane \cite{K-99}, we proved a Carlson-type density estimate for zeroes in the critical strip, for which we needed two additional assumptions that time. However, very recently \cite{Rev-NewD} we succeeded in obtaining a more general version of the Carlson-type density result, which required only Axiom A to hold. Below we will present the exact formulation in Theorem \ref{th:NewDensity}. Let us note here that--at about the same time and by entirely different methods--Broucke and Debruyne obtained a similarly general version \cite{BrouckeDebruyne}. Although their exponent is way better than ours, this has no significance for us here, while our formulation has the slight advantage of being quite explicit about the arising constants.

Finally, in \cite{Rev-One} we dealt with the problem of Littlewood in the Beurling case. Though its exact form is not used here, we formulate the result for comparison\footnote{About the need for an analogous but different form of the result consider the explanations after Remark \ref{rem:strongerthanPintz}, too.} with the forthcoming Theorem \ref{th:goodlocalization}, which on the other hand does indeed play a crucial role in our argument. Here and everywhere in the sequel we denote by $A_0,A_1,\ldots$ constants \emph{depending explicitly on the main parameters}\footnote{That is, depending--in an explicit form--only on $\theta, \kappa$ and $A$ from Axiom A.} of the Beurling system $\GG$ only.
\begin{theorem}\label{th:Bonezeroosci-plus} Let $\ze(\rho_0)=0$ with
$\rho_0=\beta_0+i\gamma_0$ and $\beta_0>\theta, ~ \gamma_0>0$.
Then for arbitrary $0<\ve<0.1$ and $\log Y > Y_0(\ve,\rho_0):=\max \left\{ \frac{5\log\frac{1}{\beta_0-\theta}}{\beta_0-\theta}, \frac{\log(8/\ve)}{\beta_0-\theta}, \frac{40}{\ve^2 \gamma_0^4}, \log|\rho_0|, A_1 \right\},$ there exists an $x$ in the interval
\begin{equation}\label{eq:xinterval}
I:=\left[Y,Y^{A_2\frac{\log(\gamma_0+5)}{(\beta_0-\theta)^2}}\right],
\end{equation}
such that
\begin{equation}\label{eq:oscillationpihalf}
\left| \D (x) \right| > \left(\frac{\pi}{2}-\varepsilon\right)
\frac{x^{\beta_0}}{|\rho_0|}.
\end{equation}
\end{theorem}

Interestingly, the best constant in \eqref{eq:oscillationpihalf} is indeed $\pi/2$. To construct counterexamples, however, required the full strength of the recent sharpening \cite{BrouckeVindas} by Broucke and Vindas of the Diamond, Montgomery and Vorhauer method.

Here we generalize the above results of Pintz in the direct and converse Ingham problem to the Beurling case, assuming Axiom A only.
\begin{theorem}\label{th:Beurlingdomainesti} Let the arithmetical semigroup $\GG$ satisfy Axiom A. Assume that $\eta(t):[0,\infty)\to [0,1-\theta]$ is an arbitrary real function such that $\DD(\eta)$ is free of zeroes of the Beurling zeta function $\zp$.

Then for arbitrary $\ve>0$ we have $\Dp(x)=O(x\exp(-(1-\ve)\omega_\eta(x)))$.
\end{theorem}
\begin{theorem}\label{th:Beurlingdomainosci}
Conversely, let us assume that the arithmetical semigroup $\GG$ satisfies Axiom A, and that there are infinitely many zeroes of $\zp$ within the domain \eqref{eq:etazerofree}, where $\eta(t)$ is convex in logarithmic variables (i.e. $\eta(e^v)$ is convex).

Then we have for any $\ve >0$ the oscillation estimate $\Dp(x)=\Omega(x\exp(-(1+\ve)\omega_\eta(x)))$.
\end{theorem}
\begin{remark}\label{rem:strongerthanPintz} This is slightly stronger than Theorem \ref{th:domainosci} of Pintz, for he requires $\eta'(t)$ increasing, while here we need only $\frac{d}{dv} (\eta(e^v))=\eta'(e^v) e^v =\eta'(t)\cdot t$ increasing. In the most interesting special cases of $\eta(t)=c\log^p(\log t) \log^q t$, with $q<0$ and $p\in \RR$, these conditions are both satisfied, however.
\end{remark}
On our way towards answering the Ingham question, we will need a result similar to Theorem \ref{th:Bonezeroosci-plus} but with a much better localization for the occurring value of $x$. Basically, instead of \eqref{eq:xinterval} we need an interval of the type $[X^{1-\ve},X]$. That we can get at the expense of some restrictions on the zero considered and a slight loss in the magnitude of the obtained oscillation, see Theorem \ref{th:goodlocalization} below. This corresponds to Theorem 2 of \cite{Pintz1} in the classical case.

In fact, we will prove some even stronger statements than the above Theorems \ref{th:Beurlingdomainesti} and \ref{th:Beurlingdomainosci}, whose formulation needs further explanations and technicalities. Therefore, these further results will be discussed in due course in the sequel only, see in particular the end of Sections \ref{sec:upper} and \ref{sec:parametersandproof}. These correspond to recent\footnote{Note that these surfaced decades later than his above cited original achievements.} advances of Pintz \cite{PintzProcStekl} in the classical case. However, to see that these modern versions are indeed stronger is not equally easy for the two directions. For the direct estimate of $\D(x)$ from knowledge about the zeroes, the comparison is almost trivial, but in the direction of the inverse, oscillation results, it becomes somewhat involved. This is explained by the different formulations: location of the zeroes is directly estimated if a function $\eta(t)$ defines a zero-free region \eqref{eq:etazerofree}, but in case we assume that \emph{there are infinitely many zeroes in the domain} $D(\eta)$, it is not clear how the distribution of zeroes can be compared to the curve bounding $D(\eta)$. And indeed, there is no automatic, general comparison of Theorem \ref{th:Beurlingdomainosci} and the Pintz type newer result, and a favorable comparison is possible only along an indirectly and geometrically constructed subsequence. This comparison we clarify at the end of Section \ref{sec:parametersandproof}. This comparison was not clarified earlier, not even in the classical case of natural numbers and the Riemann zeta function.

\section{Some auxiliary lemmas}\label{sec:aux}

\begin{lemma}\label{l::integralformula} For $a>0$, $b\in\CC$ and $c\in \RR$, we have
\begin{equation}\label{eq:integralformula}
\frac{1}{2\pi i} \int_{c-i\infty}^{c+i\infty} e^{as^2+bs} ds =
\frac{1}{2\sqrt{\pi a}} \exp\left(-\frac{b^2}{4a}\right).
\end{equation}
\end{lemma}

\begin{proof} This directly computable formula is taken--as in \cite{RevAA}--from \cite{Pintz1}, formula (10.2).
\end{proof}

\begin{lemma}\label{l:estimates} The following estimates hold true.
\begin{itemize}
\item{(i)} For any $B\geq 1/2$,
$$ \int_B^{\infty} e^{-x^2} dx < e^{-B^2}.$$
\item{(ii)} For any $0\le \lambda \le 2$ and $B\ge 1$ we have
$$ \int_B^{\infty} x^{\lambda} e^{-x^2} dx < B^{\lambda-1}e^{-B^2}.$$
\item{(iii)} For any $\lambda \geq 1$, $0<\alpha <1$ and $x\geq 1$ we have
$$ \log^{\lambda } x \leq e^{\lambda /\alpha+\lambda ^2} x^{\alpha}.$$
\end{itemize}
\end{lemma}

The proofs are easy direct calculations. For (i) and (iii) see also \cite{RevAA}.

\begin{lemma}[{\bf Continuous form of the Second Main Theorem of Turán's Power Sum Theory}]\label{l:SecondMain} If $0< H < K$, $n\in \NN$, and
$w_{j}\in\CC$ ($j=1,\dots,n$) are arbitrary complex numbers with $\Re w_1=0$, then we have
$$
\max_{H\leq M \leq K} \Re \left( \sum_{\ell=1}^{n} e^{i w_{\ell} M} \right) \geq \left(\frac{K-H}{8eK }\right)^n.
$$
\end{lemma}

For the proof see \cite{Turan} or \cite{SosTuran}.

\section{Auxiliary results on the Beurling $\zeta$ function}\label{sec:basics}

Here we collect some basic estimates and technical lemmas on the behavior of the Beurling $\zeta$ function.
Most of them are well-known, see, e.g., \cite{Knopf} or \cite{Beke} or \cite{DZ-16}. In \cite{Rev-MP} we elaborated on their proofs only for the explicit handling of the arising constants in these estimates.

However, the Riemann-von Mangoldt formula in Proposition \ref{prop:vonMangoldt} was first given in \cite{Rev-MP}, and a Carlson-type density theorem for the Beurling zeta function was first proved--under two extra assumptions--in \cite{Rev-D}. In this regard, recent development was rather fast, with general forms (assuming solemnly Axiom A) of the Carlson type density estimate appearing in \cite{Rev-NewD} and also in \cite{BrouckeDebruyne}. We will give in Theorem \ref{th:NewDensity} the version from the former paper.

\subsection{Estimates for the number of zeroes of  $\zeta$}\label{sec:zeroes}

In general in this paper we will use with arbitrary $a,\alpha \in [0,1)$ and $0<R<T$ the notations
\begin{align}\label{eq:Zsetdef}
\Z(a;~T) \quad&:=\{\rho=\beta+i\gamma~:~ \Re \rho=\beta\ge a, |\gamma|\le T\}, \notag
\\ \notag \Z(a,\alpha;~T)\quad&:=\{\rho=\beta+i\gamma~:~ \Re \rho=\beta \in [a,\alpha), |\gamma|\le T\},
\\ \Z(a;~R,T)&:=\{\rho=\beta+i\gamma~:~ \Re \rho=\beta\ge a, R<|\gamma|\le T\},
\\ \notag \Z(a,\alpha;~R,T)&:=\{\rho=\beta+i\gamma~:~ \Re \rho=\beta \in [a,\alpha), R<|\gamma|\le T\}.
\end{align}
That is, e.g., $\Z(a;T)$ and $\Z(a;R,T)$ are the set of Beurling zeta zeroes in the rectangle $[a,1]\times [-iT,iT]$ and $[a,1]\times [-iT,-iR)\cup (iR,iT]$, respectively. Their number plays an essential role in the study of the Beurling zeta function: the standard notations for them are $N(a,T):=\# \Z(a;T)$ and $N(a;R,T):=\#\Z(a;R,T)$.

\begin{lemma}\label{l:Littlewood} Let $\theta<b<1$ and consider
any height $T\geq 5$.
Then the number of zeta-zeroes $N(b,T)$ in the rectangle $\Z(b;T)$ satisfies
\begin{equation}\label{zeroesinth-corr}
N(b,T)\le \frac{1}{b-\theta}
\left\{\frac{1}{2} T \log T + \left(2 \log(A+\kappa) + \log\frac{1}{b-\theta} + 3 \right)T\right\}.
\end{equation}
\end{lemma}
\begin{proof} See Lemma 3.5 in \cite{Rev-MP}. \end{proof}

\begin{lemma}\label{l:zeroesinrange}
Let $\theta<b<1$ and consider any heights $T>R\geq 5 $.

Then the number of zeta-zeroes $N(b,R,T)$ in the rectangle $\Z(b,R,T)$ satisfies\footnote{This statement and the results of the forthcoming Lemmas \ref{l:path-translates} and \ref{l:zzpongamma-c} are slightly corrected as compared to the versions stated in \cite{Rev-MP} in view of a calculation error detected in their original proofs. The corrections result only in the change of some coefficients and were explained in somewhat more detail in \cite{Rev-D}.}
\begin{equation}\label{zeroesbetween}
N(b,R,T) \leq\frac{1}{b-\theta} \left\{ \frac{4}{3\pi} (T-R) \left(\log\left(\frac{11.4 (A+\kappa)^2}{b-\theta}T\right)\right)  + \frac{16}{3}  \log\left(\frac{60 (A+\kappa)^2}{b-\theta}T\right)\right\}.
\end{equation}

In particular, for the zeroes between $T-1$ and $T+1$ we have for
$T\geq 6$
\begin{align}\label{zeroesbetweenone}
N(b,T-1,T+1) \leq \frac{1}{(b-\theta)} \left\{ 6.2 \log T +
6.2 \log\left( \frac{(A+\kappa)^2}{b-\theta}\right) + 24 \right\}.
\end{align}
\end{lemma}
\begin{proof} This is a corrected version of Lemma 3.6 from \cite{Rev-MP}, corrected in Lemma 6 of \cite{Rev-D}, with the slight correction explained in the footnote on page 1052. \end{proof}

\subsection{The logarithmic derivative of the Beurling $\zeta$}
\label{sec:logder}

\begin{lemma}\label{l:borcar} Let $z=a+it_0$ with $|t_0| \geq e^{5/4}+\sqrt{3}=5.222\ldots$ and $\theta<a\leq 1$. With $\delta:=(a-\theta)/3$ denote by $S$ the
(multi)set of the $\ze$-zeroes (listed according to multiplicity)
not farther from $z$ than $\delta$. Then we have
\begin{align}\label{zlogprime}
\left|\frac{\ze'}{\ze}(z)-\sum_{\rho\in S} \frac{1}{z-\rho}
\right| & < \frac{9(1-\theta)}{(a-\theta)^2}
\left(22.5+14\log(A+\kappa)+14\log \frac{1}{a-\theta} + 5\log |t_0|\right).
\end{align}

Furthermore, for $0 \le |t_0| \le 5.23$ an analogous estimate (without any term containing $\log |t_0|$) holds true:
\begin{equation}\label{zlogprime-tsmall}
\left|\frac{\ze'}{\ze}(z)+\frac{1}{z-1}-\sum_{\rho\in S} \frac{1}{z-\rho}
\right|  \le \frac{9(1-\theta)}{(a-\theta)^2}
\left(34+14\log(A+\kappa)+18\log \frac{1}{a-\theta}\right).
\end{equation}
\end{lemma}
\begin{proof} See Lemma 4.1 of \cite{Rev-MP}. \end{proof}

\begin{lemma}\label{l:path-translates} For any given parameter
$\theta<b<1$, and for any finite and symmetric to zero set $\A\subset[-iB,iB]$ of cardinality $\#\A=n$, there exists a broken line $\Gamma=\Gamma_b^{\A}$, symmetric to the real axis and consisting of horizontal and vertical line segments only, so that its upper half is
$$
\Gamma_{+}= \bigcup_{k=1}^{\infty}
\{[\sigma_{k-1}+it_{k-1},\sigma_{k-1}+it_{k}] \cup
[\sigma_{k-1}+it_{k},\sigma_{k}+it_{k}]\},
$$
with $\sigma_j\in [\frac{b+\theta}{2},b]$, ($j\in\NN$),  $t_0=0$, $t_1\in[4,5]$ and $t_j\in
[t_{j-1}+1,t_{j-1}+2]$ $(j\geq 2)$ and satisfying that the
distance of any $\A$-translate $\rho+ i\alpha ~(i\alpha\in\A)$ of a $\zeta$-zero $\rho$ from any point $s=t+i\sigma \in \Gamma$ is at least $d:=d(t):=d(b,\theta,n,B;t)$ with
\begin{equation}\label{ddist-corr}
d(t):=\frac{(b-\theta)^2}{4n \left(12 \log(|t|+B+5) + 51 \log (A+\kappa) + 31 \log\frac{1}{b-\theta}+ 113\right)}.
\end{equation}
Moreover, the same separation from translates of $\zeta$-zeroes holds also for the
whole horizontal line segments $H_k:=[\frac{b+\theta}{2}+it_k,2+it_k]$, $k=1,\dots,\infty$, and their reflections $\overline{H_k}:=[\frac{b+\theta}{2}-it_k,2-it_k]$, $k=1,\dots,\infty$, and furthermore the same separation holds from the translated singularity points $1+i\al$ of $\zeta$, too.
\end{lemma}
\begin{proof} This is proved as Lemma 4.2 in \cite{Rev-MP}, and is corrected slightly in Lemma 8 of \cite{Rev-D} as a consequence of the mentioned necessary correction concerning Lemma \ref{l:zeroesinrange} above. \end{proof}

\begin{lemma}\label{l:zzpongamma-c} For any $0<\theta<b<1$ and symmetric to $\RR$ translation set $\A\subset [-iB,iB]$, on the broken line $\Gamma=\Gamma_b^{\A}$, constructed in the above Lemma \ref{l:path-translates}, as well as on the horizontal line segments $H_k:=[a+it_k,2+it_k]$ and  $\overline{H_k}$, $k=1,\dots,\infty$ with $a:=\frac{b+\theta}{2}$, we have uniformly for all $\alpha \in \A$
\begin{equation}\label{linezest-c}
\left| \frac{\ze'}{\ze}(s+i\alpha) \right| \le n
\frac{1-\theta}{(b-\theta)^{3}} \left(10 \log(|t|+B+5)+60\log(A+\kappa) + 42 \log\frac1{b-\theta}+ 140\right)^2.
\end{equation}
\end{lemma}
\begin{proof} Compare Lemma 4.3 of \cite{Rev-MP} and Lemma 9 of \cite{Rev-D}, where the above mentioned necessary corrections of the numerical constants are implemented. \end{proof}

\subsection{A Riemann-von Mangoldt type formula of prime distribution
with zeroes of the Beurling $\zeta$}\label{sec:sumrho}

We denote the set of $\zeta$-zeroes, lying to the right of $\Gamma$, by $\Z(\Gamma)$, and denote $\Z(\Gamma,T)$ the set of those zeroes $\rho=\beta+i\gamma\in \Z(\Gamma)$ which satisfy $|\gamma|\leq T$.

\begin{proposition}[Riemann--von Mangoldt formula]\label{prop:vonMangoldt}
Let $\theta<b<1$ and $\Gamma=\Gamma_b^{\{0\}}$ be the curve defined in Lemma \ref{l:path-translates} for the one-element set $\A:=\{0\}$ with $t_k$ denoting the corresponding set of abscissae in the construction.
Then for any $k=1,2,\ldots$, and $4 \leq t_k$ we have
\begin{equation}\label{eq:Riemann-vonMangoldt}
\psi(x)=x - \sum_{\rho \in \Z(\Gamma,t_k)}
\frac{x^{\rho}}{\rho} + O\left( \frac{1-\theta}{(b-\theta)^{3}}
\left(A+\kappa+\log \frac{x+t_k}{b-\theta}\right)^3 \left(\frac{x}{t_k} + x^b \right) \right),
\end{equation}
with the implied $O$-constant an effective, absolute constant which does not depend on $\GG$.
\end{proposition}
\begin{proof} See Theorem 5.1 of \cite{Rev-MP} as is corrected in Lemma 10 of \cite{Rev-D}, see also the remarks following the formulation in the latter paper. \end{proof}

\subsection{A density theorem for $\ze$-zeroes close to the $1$-line}
\label{sec:density}

As told above, regarding the density estimates for the Beurling zeta function the first notable result is due to Kahane \cite{K-99}, who proved an $O(T)$ estimate for the number of zeroes precisely lying on some vertical line $\Re s=a$. Unexpectedly, a Carlson type zero density estimate surfaced in \cite{Rev-D} under two additional assumptions on the Beurling system. Even more surprisingly, these additional assumptions could later be removed \cite{Rev-NewD}, \cite{BrouckeDebruyne}. Our version in \cite{Rev-NewD} is fully explicit in handling all constants, and is suitable to get effective, explicit results of prime distribution, while the other work \cite{BrouckeDebruyne} obtained a much better exponent than either \cite{Rev-D} or \cite{Rev-NewD}. Let us point out that the value of this exponent is important for much of the related number theory, in particular for generalizing the analysis of primes in short intervals, as is impressively worked out in \cite{BrouckeDebruyne}. In the applications in our present work the exponent has no real significance, however. What matters is the fact that a Carlson type density estimate does hold with some exponent\footnote{It is worth noting that already Diamond, Montgomery and Vorhauer argued \cite{DMV} that anything essentially better does not hold in general (i.e. assuming only Axiom A), while Broucke and Debruyne \cite{BrouckeDebruyne} sharpened their example to get a uniform lower estimate $\log N(\alpha,T) \gg \frac{1-\alpha}{1-\theta}\log T $.}. Below we recall the main result of \cite{Rev-NewD}, see Theorem 2 in that paper.

\begin{theorem}\label{th:NewDensity} Let $\G$ be a Beurling system subject to Axiom $A$. Then for any $\si>(1+\theta)/2$ the number of zeroes of the corresponding Beurling zeta function $\zs$ admits a Carlson-type density estimate
\begin{equation}\label{densityresult}
N(\si,T) \le  1000 \frac{(A+\kappa)^4}{(1-\theta)^3 (1-\si)^4} T^{\frac{12}{1-\theta}(1-\si)} \log^5 T
\end{equation}
for all $T\ge T_0$, where also $T_0$ depends explicitly on the parameters $A, \kappa, \theta$ of Axiom A and on the value of $\si$. In particular, for $\si>\frac{11+\theta}{12}$ we have $N(\si,T)=o(T)$.
\end{theorem}

Note that we separated terms of powers of $\log T$ and powers of $1/(1-\si)$ only for the explicit handling of constants. A classical zero-free region of the form \eqref{classicalzerofree} always holds, but the term $1/(1-\si)^4$ changes the constant factor by $c^{-4}$ from \eqref{classicalzerofree}. When applying Theorem \ref{th:NewDensity} in the present work it will always suffice to use the estimate $N(\si,T) \ll_{A,\theta,\kappa,c}  T^{\frac{12}{1-\theta}(1-\si)} \log^9 T \ll_{A,\theta,\kappa,c,\ve} T^{\frac{12}{1-\theta}(1-\si)+\ve}$ only.

\subsection{A Tur\'an type local density theorem for $\ze$-zeroes close
to the boundary of the zero-free region} \label{sec:localdensity}

Recall that for arbitrary $\tau>0$ and $\theta<\si<1$, we denoted the set of zeroes in the rectangle $Q_{\sigma,h}(\tau):=[\si,1]\times[i(\tau-h),i(\tau+h)]$ as $\Z(\si;\tau-h,\tau+h)$, and their number as $N(\si,\tau-h,\tau+h)$.

\begin{lemma}\label{th:locdens} Let $(1+\theta)/2 < b\leq 1$,
$2\leq h$ and $\tau>\max(2h,\tau_0)$ where
$\tau_0=\tau_0(\theta,A,\kappa)$ is a large constant depending on
the given parameters of $\zs$.

If $\zs$ does not vanish in the rectangle $\sigma\geq b$,
$|t-\tau|\leq h$, that is if $\Z(b;\tau-h,\tau+h)=\emptyset$, then for any $r$ with
$15\frac{\log \log \log \tau}{\log \log \tau} < r < \frac{b-\theta}{10}$ we have
\begin{equation}\label{eq.TuranLemma}
\nu:=N(b-r;\tau-r,\tau+r) \ll r \log\tau,
\end{equation}
with an implied absolute constant not depending on $\GG$.
\end{lemma}

\section{Upper estimate on $\D(x)$--Proof of Theorem \ref{th:Beurlingdomainesti}}\label{sec:upper}

This section is devoted to the proof of Theorem \ref{th:Beurlingdomainesti} along with some more general statements. To start with, let us define
$$
\theta_0:=\max(\theta,\sup\{\theta<\beta<1 ~:~ \exists \gamma\in \RR, ~\textrm{such that}~ \zeta(\beta+i\gamma)=0\}).
$$
Below in \eqref{eq:Womegadef} a new auxiliary function $\om$ will be introduced, and it will be easily seen that under the assumptions of Theorem \ref{th:Beurlingdomainesti}--that is that $D(\eta)$ is free of zeroes of $\zs$--we have $\om_\eta(x)\le \omega(x)$. Also, in Lemma \ref{l:omega} we will find $\omega(x) \sim (1-\theta_0)\log x$ if $\theta<\theta_0\le 1$. These show us that in case $\theta_0<1$ it suffices to prove$\D(x)=O(x^{\theta_0+\ve})$ for all $\ve>0$. So we will start here with explaining how this follows easily from the Riemann-von Mangoldt formula in Proposition \ref{prop:vonMangoldt} for $\theta_0<1$.

First, if $\theta_0=\theta$, then we choose $b:=\theta_0+\ve/2$. According to the construction in Lemma \ref{l:path-translates} there is a curve $\Gamma:=\Gamma_b^{\{0\}}$ whose points stay in the strip $a:=\frac{b+\theta}{2}\le \Re s \le b$: in view of the assumption $\theta_0=\theta$ on this curve and to the right of it there are simply no zeroes of $\ze$. Therefore, the Riemann-von Mangoldt type formula of Proposition \ref{prop:vonMangoldt} reduces to the error term, so that with $t_k$ of the order $x$ we get $\Delta(x)=O_\ve(\log^3 x~ x^b)=O_\ve(x^{\theta_0+\ve})$, as needed.

For other intermediate values $\theta<\theta_0<1$, we either work out the analogous constructions of Lemma \ref{l:path-translates} and Proposition \ref{prop:vonMangoldt} with $a':=\frac{b+\theta_0}{2}$, and then the proof is the same, with no zeroes to the right of $\Gamma'$, or use the below argument. Regarding the first possibility, we can remark that as long as $\GG$ satisfies Axiom A with $\theta$ as the main parameter, it also satisfies Axiom A with any $\theta'\ge \theta$ in place of it; so that we can refer to the said Lemmas, but used with $\theta':=\theta_0$ instead. That is, we can reduce the case to the already settled one of $\theta_0=\theta$ via setting $\theta'=\theta_0$.

\bigskip
Nevertheless, in the following we will describe here the direct argument, for two reasons. First, in some later calculus--in particular for the case of $\theta_0=1$--the proof will as well be analogous. Moreover, here we want to prove some sharper statements, too. For their formulation, analogously to Ingham's function \eqref{eq:omegaetadef} and following Pintz\footnote{The function $\om(x)$ was introduced in \cite{Pintz2} for the case of the Riemann $\zeta$ function.} we define for any $a>\theta$ and $x\ge 1$ the functions\footnote{Note that we do not consider the Beurling zeta function on and behind the line $\Re s=\theta$, where a meromorphic (or at least some meaningful, e.g. continuous) extension of $\zs$ is not guaranteed by Axiom A. Whenever we speak of $\zeta$-zeroes, we always mean zeroes with $\beta=\Re \rho >\theta$.}
\begin{align}\label{eq:Womegadef}
\omega_a(x)&:=\inf_{\zeta(\rho)=0, \rho=\beta+i\gamma \atop a<\Re \rho=\beta} \log\frac{x}{x^\beta/|\rho|}=\inf_{\rho\in\Z(a;\infty)} (1-\beta)\log x+ \log|\rho|, \notag
\\ \omega(x)&:=\inf_{\zeta(\rho)=0, \rho=\beta+i\gamma \atop \theta<\Re \rho=\beta} \log \frac{x}{x^\beta/|\rho|}=\inf_{\rho\in\Z(\theta;\infty)} (1-\beta)\log x+ \log|\rho|,
\\ \ot(x)&:=\inf_{\zeta(\rho)=0, \rho=\beta+i\gamma \atop \theta<\Re \rho=\beta, \gamma>1} \log \frac{x}{x^\beta/\gamma}=\inf_{\rho\in\Z(\theta;\infty)\setminus\Z(\theta,1)} (1-\beta)\log x+ \log\gamma.  \notag
\end{align}

A standard analysis of a few properties of $\om_\eta$ and these analogous $\omega$ functions is in order here. We could have described much of the properties via use of the Legendre transform\footnote{Ingham or Pintz did not mention the Legendre transform in their work. To the best of our knowledge, this notion of convex analysis surfaced in connection with the $\zeta$ function in \cite{Rev-L} and later also in \cite{H-19}, while the exact mention of these $\omega$-functions as being Legendre tranforms first appeared in \cite{Rev-D}.}, see footnote 5, but for being self-contained we instead work out all details (even if it was there in the back of our mind when figuring out the presentation of our elementary description).

Let us start with $\om_\eta$, which was defined by Ingham for the case of a continuous nonincreasing function $\eta$. With a slight inconsistency it was already used in the formulation of Theorem \ref{th:Beurlingdomainesti} for an arbitrary real function $\eta: [0,\infty) \to [0,1-\theta]$, where, to be fully precise, formally we should have written $\inf$ in place of the $\min$ in the definition, similarly to the above variants.

However, one can always replace $\eta$ by its lower semicontinuous envelope $\eta_*(t):=\liminf_{\tau \to t} \eta(\tau)$, without change of the infimum; and the infimum becomes a minimum once we consider a lower semicontinuous function (like the lower semicontinuous envelope itself), because then $\eta(t)\log x +\log t$ is lower semicontinuous, too. Therefore, this slight inconsistency can always be overcome with the agreement that we always consider $\eta$ lower semicontinuous if need be.

\begin{lemma}\label{l:omegaeta} For an arbitrary real function $\eta: [0,\infty) \to [0,1-\theta]$ the respective $\oet$ function remains the same if we replace $\eta$ by its lower semicontinuous envelope $\eta_*$, and even if $\eta_*$ is again replaced by its "monotone nonincreasing cover" $\widetilde{\eta}(t):=\min_{1\le \tau \le t} \eta(\tau)$.

Moreover, $\oet$ remains the same if we further replace the "\emph{log-convex envelope}" $\widehat{\eta}(e^v):=\sup \{f(v)~: \linebreak ~f (v)\le \eta(e^v),~f~ \rm{convex}\}$ for $\widetilde{\eta}(e^v)$.
\end{lemma}

\begin{remark}\label{rem:zero-one} Before proceeding, let us point out another slight inconsistency in talking about the various transforms "of the function $\eta$", while in reality we use and transform only the restriction $\eta|_{[1,\infty)}$ in all our calculus, in particular in defining and using $\oet$. This did not bother Ingham or Pintz, as the first zero of the Riemann zeta function occurs above imaginary part 14 anyway; but causes us some technical inconveniences here and there. Still, we need to stick to considering only $t\ge 1$, i.e. $v:=\log t\ge 0$--or at least $v$ bounded from below--for considering $\eta(e^v)$ all over $\RR$ would necessarily invoke $v=\log t$ values arbitrarily close to $-\infty$, ruining the minimization procedure in the definition of $\oet$.
As a result, also the convex envelopes and Legendre transform interpretations would be destroyed by extending $t$ arbitrarily close to 0, a curious technical problem which we had to avoid. A similar technicality is reflected in the need for defining, besides the original $\oet$ of Ingham and $\om$ of Pintz, also $\om_a$ and $\widetilde{\om}$ above.
\end{remark}
\begin{proof} We have already told about $\eta_*$, therefore let us assume that $\eta$ is lower semicontinuous.

So with a given $x$ let $t$ be such that $\oet(x)=\eta(t)\log x + \log t$. By definition of $\oet(x)$ we have for all $1\le \tau <t$ the inequality $\eta(\tau) \log x +\log \tau \ge \oet(x)$, hence $(\eta(\tau)-\eta(t))\log x \ge \log t - \log \tau >0$. This does not necessarily mean that $\eta$ itself is monotone decreasing (as certain $t$ may \emph{not} occur as minimum points for some $x$), but it follows that replacing $\eta$ by $\eta_*$ changes the values of $\om(x)$ for no $x$.

Write now $g(v):=\eta(e^v)$.
Then we claim that $\om_\eta(e^u):=\inf_v g(v) u + v$ equals to $\widehat{\om}(e^u):=\om_{\widehat{\eta}} (e^u):=\inf_v \widehat{g}(v) u +v$, where $\widehat{g}(v):=\widehat{\eta}(e^v):=\sup\{ f(v)~:~ f\le g, ~f~\rm{convex} \}$. The inequality $\widehat{g}(v)\le g(v)$ is obvious, hence $\widehat{\om} (e^u) \le \om_\eta(e^u)$. Let now
$f(v):= \om_\eta(e^u)/u - (1/u)v$; then rearranging the defining formula for $\om_\eta(e^u)$ we get that this linear, hence convex function satisfies $f(v)\le g(v)$ for all $v$. It follows that $\widehat{\om} (e^u) :=\min_v \widehat{g}(v)u+v \ge \min_v uf(v)+v =\om_\eta(e^u)$, too.
\end{proof}

\begin{remark} If $\eta$ was not bounded, then it could be possible that the monotone nonincreasing envelope essentially changes it; more precisely, that the final outcome of taking $\widehat{\widetilde{\eta}}$ would stay below the function $\widehat{\eta}$ and also the corresponding $\om$ functions would deviate. However, in our setup $\eta$ is bounded, hence no eventually increasing convex function $f$ can stay below $\eta(e^v)$, and as a result, it is easy to see that $\widehat{\eta}$ will itself be nonincreasing. Similarly, it must be continuous. Therefore, taking $\eta_*$ and then $\widetilde{\eta}$, and applying the convex envelope only after, or taking the convex envelope immediately to $\eta$ results the same function $\widehat{\eta}$. All these and more are well explained in convex analysis related to the Legendre transform, see, e.g., \cite{Rockafellar}.
\end{remark}

\begin{lemma}\label{l:Ingham} If $\eta:[0,\infty] \to [0,1-\theta]$ is an arbitrary real function, then $\oet$ is a strictly increasing to $+\infty$ continuous function and such that $\log x- \oet(x)$ is nondecreasing, too. Moreover, $\oet$ is concave in logarithmic variables. Furthermore, if $\theta^*:=1-\inf_{[1,\infty]} \eta$, then $\lim_{x\to \infty} \frac{\oet(x)}{\log x}= 1-\theta^*$.
\end{lemma}
\begin{proof} In view of the above Lemma \ref{l:omegaeta} $\oet(x)=\om_{\widetilde{\eta}}(x)$, hence we can restrict to the case when $\eta$ itself is continuous and nonincreasing. For this case Ingham has already derived an elementary way all the stated properties in the first sentence--see the turn of pages 63 and 64 of \cite{Ingham} for this elementary argument.

Concavity follows from the fact that $\oet(e^u)$ can be written in terms of the Legendre transform as $-{\mathcal L}(\phi)(u)$, see footnote 4, but we do not need to refer to the Legendre transform here, as concavity is an easy direct fact. Indeed, $\oet(e^u)$ is defined as the infimum of a family of concave (actually: linear) functions, and as such, must be concave itself.

Lastly, let us prove the assertion about the limit. If $\oet(x)=\eta(t)\log x + \log t$, then $\oet(x)\ge \eta(t) \log x \ge (1-\theta^*) \log x $. For the other direction, consider an arbitrary (small) $\ve>0$ and a value $t'$ with $\eta(t')<1-\theta^*+\ve$; then $\oet(x) \le \eta(t') \log x + \log t' \le (1-\theta^* +\ve) \log x + \log t'$. Dividing these two inequalities by $\log x$ and then taking limits we obtain $\lim_{x \to \infty} \frac{\oet(x)}{\log x} =1-\theta^*$, as stated.
\end{proof}

\begin{lemma}\label{l:varyingofomegaeta} If $\eta:[0,\infty] \to [0,1-\theta]$ is a real function, then $\oet$ admits the following properties.
\begin{enumerate}
\item For any $\ve>0$ and large enough $x>x_0(\ve)$ for all $\sqrt{x}< y<x$ we have
$$
\oet(x)-(1-\theta^*+\ve) \log(x/y)\le \oet(y)\le \oet(x).
$$
\item If $\theta^*=1$, then $\oet(x)$ is a slowly varying function in the sense that
\begin{equation}\label{eq:slowlyvarying}
\left|\oet(y)-\oet(x)\right| \le \ve |\log(y/x)| \qquad (\sqrt{x} \le y\le x^2, ~x>x_0(\ve)).
\end{equation}
\item If $1 <y < z$ then $\oet(y) < \oet(z) < \oet(y) \cdot \frac{\log z}{\log y}$.
\end{enumerate}
\end{lemma}
\begin{proof} Assume, as we may in view of Lemma \ref{l:omegaeta}, that $\eta$ is a nonincreasing, continuous, convex in logarithmic variables function.

By definition of $\oet$, for any fixed value of $t\ge 1$ we have $\oet(x) \le \eta(t)\log x + \log t=:\ff_t(x)$, for all $x$. Now obviously for any $\tau \in [1,\infty)$ the inequality $\ff_\tau(x) \le \ff_t(x)$ requires that also $\log\tau \le \ff_\tau(x)$ is below $\ff_t(x)$, hence $\tau$ stays bounded and the infimum defining $\oet(x)$ has to be a minimum over values on a compact interval. In particular, $\oet(x)$ is always attained by some respective $\ff_t(x)$, so that we may introduce the notation $t(x):=\min\{ t\ge 1~:~\ff_t(x)=\oet(x)\}$.

We claim that $t(x)$ is a nondecreasing (even if not necessarily continuous) function of $x$. Indeed, minimality of $t_0$ for a certain given value of $x_0$ in the definition of $\oet(x_0)$ means $\oet(x_0)=\ff_{t_0}(x_0) = \eta(t_0)\log x_0 + \log t_0 \le \ff_t(x_0)=\eta(t)\log x_0 + \log t$ for all $t\ge 1$. Let us rewrite this by putting $\xi:=\log x_0, v_0:=\log t_0$, $u_0:=g(v_0):=\eta(e^{v_0})$, $v:=\log t$, $u:=g(v):=\eta(e^v)$: we get $L(u,v):=\xi(u-u_0)+(v-v_0)\ge 0$, for all points $(v,u)$ lying on the graph $\Gamma$ of the function $g(v)$. Now $L=0$ is the equation of a straight line $\ell$ on the $(v,u)$ plane, which is satisfied by the point $(v_0,u_0)$. That is, the point $(v_0,g(v_0))\in \Gamma$ lies on $\ell$. Further, the above inequality expresses that the line $\ell$ passes below (not above) the graph $\Gamma$, so that it is a supporting line to it. Now, the slope of this supporting line is $-1/\xi$, as $\ell$ can be described by an equation of the form $u=(-1/\xi) v + C$.

That means that for a given value of $x$ finding a minimizing $\ff_t$ in the definition of $\oet$ is equivalent to find a corresponding value $v=\log t$ with $(v,g(v))$ admitting a supporting line of slope $-1/\log x$. Therefore, concavity of $g(v)$ entails that as $-1/\log x$ increases together with $x$, also the $v=\log t$ coordinates of the corresponding tangent points (for supports of the given slopes) must increase.

Let us define $t^*:=\min\{t\ge 1~:~ \eta(t)=1-\theta^*\}$ (if it is finite--and if not, then $+\infty$). Next we show that $t(x)\to t^*$ when $x\to \infty$. We have already proved monotonicity of $t(x)$, and it is also obvious that $t(x)\le t^*$, for all $x$, because a value $t>t^*$ (if such a value exists at all, i.e., if $1-\theta^*$ is attained at some finite point and hence there are $t>t^*$, where then by monotonicity we must have $\eta(t)=1-\theta^*$, too) does never qualify. Indeed, then $t^*<t$ and $\eta(t)=\eta(t^*)=1-\theta^*$ implies 
$\ff_{t^*}(x)<\ff_t(x)$,
thus for no $x$ can $\ff_t(x)$ be minimal. Let now take $\tau <t^*$ fixed. By definition of $t^*$, then there is $\de>0$ such that $\eta(\tau)>1-\theta^*+\de$. So, if some $t\le \tau$ is optimal for some $x$--i.e., we have $\oet(x)=\ff_t(x)$--then $\ff_t(x) \ge (1-\theta^*+\de)\log x$, while $\om(x)/\log x \to 1-\theta^* (x\to \infty)$ according to Lemma \ref{l:omegaeta}, furnishing a contradiction for large $x$.


Let now $\ve>0$ be arbitrary and $x_0(\ve)$ be such that for $x>x_0$ we have $\eta(t(x))<(1-\theta^*+\ve)$. Such an $x_0$ exists because $t(x)\to t^*$ and $\eta(t(x))\to 1-\theta^*$ as $x\to \infty$. Let now  $x>y>\sqrt{x}>x_0(\ve)$ be large, and take the point $t_0:=t(y)$ so that $\oet(y)=\ff_{t_0}(y)$. Then by the above $1-\theta^* \le \eta(t_0)<1-\theta^*+\ve$ and it follows that
$$
\oet(y)=
\eta(t_0) \log y +\log t_0 = \eta(t_0) \log y/x + \eta(t_0) \log x + \log t_0 \ge  (1-\theta^*+\ve) \log y/x + \om(x),
$$
proving the left hand side part of (1), while the right hand side inequality follows by monotonicity. From part (1) the estimate of part (2) follows immediately.

Finally, consider the last assertion, the left hand side inequality coming directly from monotonicity of $\oet$. As it was seen in Lemma \ref{l:Ingham}, $f(u):=\oet(e^u)$ is a concave function. Moreover, $f(0)=0$. Therefore, $f(u)/u$ is the slope of the chord between the graph points $(0,f(0))$ and $(u,f(u))$, and as such, it is a nonincreasing function of the variable $u\in[0,\infty)$. It follows that $f(\log y)/\log y \ge f(\log z)/\log z$, which furnishes the right hand side estimate of the statement.
\end{proof}

Now let us discuss $\ot$. In fact, it is rather similar to $\oet$, even if the minimization is discretely defined. If there is no $\zeta$-zero at all with $\gamma=\Im \rho > 1$, then the minimization is on an empty set and $\ot \equiv +\infty$. If on the other hand there are zeroes, then let us denote $q:=\inf\{\gamma>1~:~ \exists \rho \in \Z(\theta,\infty)\setminus\Z(\theta,1), \Im\rho=\gamma\}$. Similarly to $\eta(e^v)$ and $\widehat{\eta}(e^v)$, here we can consider $\lambda(v):=\min\{(1-\beta)~:~ \zeta(\beta+i\gamma)=0, e^v=\gamma=\Im \rho>1\}$, while defining $\lambda(v)=\infty$ or $1-\theta$ if $\gamma:=e^v$ is not an imaginary part of a zero. As before, we can build the left continuous nonincreasing envelope $\widetilde{\lambda}$ of $\lambda$, and then its convex in logarithmic variables envelope $\widehat{\lambda}$. Note that assuring monotonicity is equivalent to restricting in the convex envelope to non-positive slopes in terms of the allowed straight line functions--this does not change the value of the respective $\omega$-function belonging to $\lambda$ or $\widetilde{\lambda}$ or $\widehat{\lambda}$.

So now we have a function $\widetilde{\lambda}:[0,\infty)$, which is defined to be $+\infty$ (or $1-\theta)$) in $[0,\log q]$ or $[0,\log q)$ and is a convex in logarithmic variables nonincreasing continuous function in $[\log q,\infty)$ or $(\log q,\infty)$. The corresponding $\omega_{\widehat{\lambda}}$ will thus exhibit all the properties described for $\oet$. However, we also have $\ot:=\omega_\lambda=\omega_{\widehat{\lambda}}$, hence we arrive at the following.

\begin{lemma}\label{l:omegatilda} Define $\widetilde{\theta}:=\sup\{\beta:=\Re(\rho)~:~ \zeta(\rho)=0, \gamma:=\Im \rho>1\}$. If $\widetilde{\theta}=-\infty$, i.e., there are no occurring zeroes in the definition of $\widetilde{\theta}$, then $\ot(x) \equiv +\infty$. If on the other hand the zero-set in the sup is nonempty, and hence $\theta<\widetilde{\theta}\le\theta_0$, then the function $\ot$ is a finite valued, continuous, monotonically increasing, convex in logarithmic variables function such that also $\log x -\ot(x)$ is nondecreasing.

Moreover, we have $\lim_{x\to \infty} \ot(x)/\log x=1-\widetilde{\theta}$, and $\ot$
admits the following properties.
\begin{enumerate}
\item For any $\ve>0$ and large enough $x>x_0(\ve)$ for all $\sqrt{x}< y<x$ we have
$$
\ot(x)-(1-\widetilde{\theta}+\ve) \log(x/y)\le \ot(y)\le \ot(x).
$$
\item If $\widetilde{\theta}=1$, then $\ot(x)$ is a slowly varying function in the sense that
$$ 
\left|\ot(y)-\ot(x)\right| \le \ve |\log(y/x)| \qquad (\sqrt{x} \le y\le x^2, ~x>x_0(\ve)).
$$ 
\item If $1 <y < z$ then $\ot(y) < \ot(z) < \ot(y) \cdot \frac{\log z}{\log y}$.
\end{enumerate}
\end{lemma}
\begin{proof} All assertions follow from the fact that $\ot=\omega_{\widehat{\lambda}}$, except the final one, i.e. property (3), whose left hand side is still trivial due to monotonicity.

As for the right hand side, the only alteration to the previous argument is that here we do not necessarily have $\ot(0)=0$, which came from the finiteness of $\eta(1)$ in the previous argument. Here, however, we can have $\widehat{\lambda}(1)=+\infty$, if $q>1$. (Note that ${\widehat{\lambda}}$ is convex in logarithmic variables only for $t>q$ or $t\ge q$.)

Nevertheless, $\ot$ remains concave in logarithmic variables, moreover, the defining quantities $(1-\beta)\log x + \log \gamma$ are still nonnegative (positive, if $\gamma>1)$, hence it still holds that $\ot(1)\ge 0$. So let $p:=\ot(1)\ge 0$. Then with $f(u):=\ot(e^u)$ we can write --similarly to the above in view of the decrease of the slope of chords of the graph--that $\frac{f(u)-p}{u}$ is nonincreasing. Note that also $p/u$ is decreasing. Adding the respective inequalities for $u:=\log y<u':=\log z$ we are led to $f(u)/u > f(u')/u'$, as before. This proves the statement.
\end{proof}

\begin{lemma}\label{l:omega-a} Let $0\le \theta<a < 1$ be arbitrary. If $\Z(a,\infty)=\emptyset$, then $\om_a(x)\equiv + \infty$. Otherwise, the function $\om_a$ is a finite valued, continuous, monotonically increasing, convex in logarithmic variables function such that also $\log x -\om_a(x)$ is nondecreasing.

Moreover, we have $\lim_{x\to \infty} \om_0(x)/\log x=1-\theta_0$, and $\om_a$
admits the following properties.
\begin{enumerate}
\item For any $\ve>0$ and large enough $x>x_0(\ve)$ for all $\sqrt{x}< y<x$ we have
$$
\om_a(x)-(1-\theta_0+\ve) \log(x/y)\le \om_a(y)\le \om_a(x).
$$
\item If $a<\theta_0=1$, then $\om_a(x)$ is a slowly varying function in the sense that
$$ 
\left|\om_a(y)-\om_a(x)\right| \le \ve |\log(y/x)| \qquad (\sqrt{x} \le y\le x^2, ~x>x_0(\ve)).
$$ 
\item If $1 <y < z$ then $\om_a(y) < \om_a(z) < \om_a(y) \cdot \frac{\log z}{\log y} +\log \frac{1}{a} \left( \frac{\log z}{\log y} - 1\right)$.
\end{enumerate}
\end{lemma}
\begin{proof} Here we can consider the function $\mu(u):=\min\{ 1-\beta~:~ |\rho|=e^u, \Re \rho=\beta>a, \zeta(\rho)=0\}$. Then $\mu$ is defined on some finite or at most countable subset of $[\log a,+\infty)$, while for values $u \ne \log |\rho|$ for any $\zeta$-zero with $\Re \rho \ge a$ we can either take $\mu(u):=+\infty$ or even $ß\mu(u):=1-a$.

Then the same procedure as above furnishes the lower convex envelope $\widehat{\mu}(u)$ of $\mu$, and mutatis mutandis it is easy to see that $\om_a:=\om_\mu=\om_{\widehat{\mu}}$. The properties then follow from the above except for the last one, (3), the left hand side of which still remaining trivial by monotonicity.

For the right hand side, however, we need to take into account that now $\om_a(1)<0$ is well possible due to $\log a <0$. In any case, we still have $\widehat{\mu}(u) \log x + \log |\rho| \ge \log |\rho|\ge \log a$, hence also for the infimum of such expressions we have $\om_a(x)\ge \log a$. It follows that $f(u):=\om_a(e^u)+\log(1/a)$ is concave with $f(u)\ge 0$ all over $[0,\infty)$, hence the above proof works for this modified function $f$.
\end{proof}

\begin{lemma}\label{l:omega} Let $q_0:=\inf\{|\rho|~:~\zeta(\rho)=0\}$. If $q_0=0$ (and hence in particular also $\theta=0$), then $\om(x)\equiv -\infty$, and if $\Z(\theta,\infty)=\emptyset$, then $\om(x)\equiv + \infty$.

In all other cases $\om: [1,\infty)\to [\log q_0,\infty)$ is a finite valued, continuous, monotonically increasing, convex in logarithmic variables function such that also $\log x -\om(x)$ is nondecreasing.

Also, the infima in the definition of $\om(x)$ is always attained, and the zeroes with attainment tend to $\theta_0+i\infty$ with $x\to \infty$, unless $\Re s=\theta_0$ contains a zero, in which case for $x>x_0$ the particular zero with $\rho_0=\theta_0+i\gamma_0$ and minimal possible $\gamma_0$ provides the said infimum, always.

Moreover, we have $\lim_{x\to \infty} \om(x)/\log x=1-\theta_0$, and $\om$
admits the following properties.
\begin{enumerate}
\item For any $\ve>0$ and large enough $x>x_0(\ve)$ for all $\sqrt{x}< y<x$ we have
$$
\om(x)-(1-\theta_0+\ve) \log(x/y)\le \om(y)\le \om(x).
$$
\item If $\theta_0=1$, then $\om(x)$ is a slowly varying function in the sense that
$$ 
\left|\om(y)-\om(x)\right| \le \ve |\log(y/x)| \qquad (\sqrt{x} \le y\le x^2, ~x>x_0(\ve)).
$$ 
\item If $1 <y < z$ then $\om(y) < \om(z) < \om(y) \cdot \frac{\log z}{\log y} +\log \frac{1}{q_0} \left( \frac{\log z}{\log y} - 1\right)$.
\end{enumerate}
Furthermore, if $\theta_0$ is not attained\footnote{In fact, the same proof also works if $\theta_0$ is attained, but the eventually extremal $\zeta$-zero $\rho^*=\theta_0+i\gamma^*$ with minimal $\gamma^*$ has $|\rho^*|\ge 1$. However, we don't need this in the paper.}--in particular if $\theta_0=1$--then there exists an $x_0$ such that for $z>y>x_0$ we also have
\begin{equation}\label{omegayz}
\om(y)< \om(z) < \om(y) \cdot \frac{\log z}{\log y}.
\end{equation}
\end{lemma}
\begin{proof} After $\om_\eta, \ot$ and $\om_a$, the reader should have no difficulty in adapting the above arguments to this case. Let us therefore prove only the last assertion, also making use of the previously listed properties.


Since $\theta_0$ is not attained, the extremal zeroes $\rho=\rho(x)$ with $\om(x)=\om(\rho;x):=(1-\beta)\log x + \log |\rho|$ tend to $\theta_0+i\infty$, so that in particular $\log|\rho|>0$ for any extremal zero for $x>x_0$.


Let now $z>y>x_0$ and take $\rho$ be extremal for $y$: $\om(y)=\om(\rho; y):=(1-\beta)\log y + \log|\rho|$. As $\om(z)\le \om(\rho;z)$, we get $\om(z)/\om(y) \le \om(\rho;z)/\om(\rho;y)= \frac{(1-\beta)\log z +\log|\rho|}{(1-\beta)\log y +\log|\rho|} <\log z/\log y$, taking into account $\log|\rho| >0$, too. This furnishes the required inequality.
\end{proof}

In this work we will denote
\begin{equation}\label{eq:DSdef}
D(X):=\frac{1}{X}\int_1^X |\Delta(x)|dx, \qquad S(X):=\sup_{1\le x\le X} |\Delta_\PP(x)|,
\end{equation}
and also
\begin{equation}\label{eq:Wdef}
W_a(x):=\sup_{\zeta(\rho)=0, \rho=\beta+i\gamma \atop a<\Re \rho=\beta, |\gamma|<x} \frac{x^\beta}{|\rho|} , \qquad W(x):= \sup_{\zeta(\rho)=0, \rho=\beta+i\gamma \atop \theta<\Re \rho=\beta, |\gamma|<x} \frac{x^\beta}{|\rho|}.
\end{equation}
Clearly, $\om(x)=\log (x/W(x))$ and $\om_a(x)=\log (x/W_a(x))$. Further, we will write
\begin{equation}\label{eq:Sdef}
Z_a(x):=\sum_{\zeta(\rho)=0, \rho=\beta+i\gamma \atop a<\Re \rho=\beta, |\gamma|<x} \frac{x^\beta}{|\rho|} , \qquad Z(x):= \sum_{\zeta(\rho)=0, \rho=\beta+i\gamma \atop \theta<\Re \rho=\beta, |\gamma|<x} \frac{x^\beta}{|\rho|}.
\end{equation}
Note that the first sum extends over $\Z(a;x)$, and the second one over $\Z(\theta,x)$.

As an easy warm-up, let us prove a straightforward assertion with these notations.
\begin{proposition}\label{p:thetathetanull} Assume that $\GG$ satisfies Axiom A and that $\theta_0=\theta$, i.e. the Beurling zeta function has no zeroes in the halfplane of meromorphic extension $\Re s>\theta$. Then for all $a>\theta$ and $x\ge 1$ we obviously have $Z_a(x), W_a(x)=0$--in fact, also $Z(x)=W(x)=0$-- while for the error term in the PNT it holds
\begin{equation}\label{eq:thetathetanullcase}
D(x)\le S(x)\ll x^{\theta} \log^6(x+1),
\end{equation}
with an effective implied constant in the Vinogradov symbol depending only on $A$ and $\kappa$.
\end{proposition}
\begin{proof}  We refer to the Riemann-von Mangoldt formula Proposition \ref{prop:vonMangoldt} with $b:=\theta+\frac{1}{\log(x+1)}$ and make use that there are no zeroes at all. This yields
\begin{equation}\label{eq:SWa}
\Delta(x) \le \left|\sum_{\rho \in \Z(\Gamma,t_k)}
\frac{x^{\rho}}{\rho}\right| + O_{A,\kappa}\left( \frac{1}{(b-\theta)^3}\log^3\frac{2x}{b-\theta} ~x^b \right) \ll_{A,\kappa} 0+ x^\theta \log^6 (x+1).
\end{equation}
Taking maximum on the interval $[1,x]$ then gives the second part of the assertion, while $D(x)\le S(x)$ is obvious.
\end{proof}
\begin{remark} Arriving at the natural boundary of analytic investigations when we reach $\Re s=\theta$, more refined results are not available in this case. It can in principle happen that the Beurling integers do not satisfy better error terms than Axiom A with some given $\theta \in (0,1)$, but the PNT still holds with much better error terms, see e.g. \cite{DebruyneDiamondVindas}, \cite{BrouckeVindas}, \cite{Zhang7}; therefore in this generality one cannot state that e.g. $\Delta(x)$ would be $\Omega(x^{\theta-\ve})$. All our investigation is focused onto the relation between distribution of zeroes and oscillation of $\Delta(x)$; but when there are no zeroes, we only know upper estimates and cannot state anything about $\Omega$-results. More on this see Remark \ref{rem:nozero} below.
\end{remark}

\begin{remark} When $\theta_0>\theta$, and there are zeroes of $\zs$ in the critical strip, it still may be impossible to handle the limiting expressions $W(x), Z(x)$. Indeed, if $\theta=0$ and $s=0$ is a limit point of zeroes of $\zs$, then $W(x)=\infty$, and if $0\le \theta\le 1$ arbitrary, but there are too many zeroes close to the boundary line $\Re s=\theta$, then $Z(x)$ can still diverge to $+\infty$. Note that $N(a,T)$ can be of the order $\frac{1}{a-\theta} \log (\frac{1}{a-\theta}) $--at least that is the best estimate we could infer, c.f. Lemma \ref{l:Littlewood}. So instead of elaborating on various cases, in the following we will prefer considering $Z_a(x)$. Nevertheless, it should be clear that $W_a(x)\le W(x)$ and $Z_a(x)\le Z(x)$, whence in case the right hand sides reach $+\infty$, all upper estimates remain valid with them.
\end{remark}

\begin{lemma}\label{l:ZerosetZ} Let $\theta<a<\alpha \le \theta_0\le 1$, $x\ge 1$ arbitrary, and consider any subset $\Z$ of the zeroes of the Beurling zeta function in the rectangle $[a,\alpha]\times [-ix,ix]$.
Then we have
\begin{equation}\label{eq:Zestimate}
\sum_{\rho \in \Z} \frac{x^\beta}{|\rho|} \ll \frac{ x^\alpha}{a-\theta} ~ \left( \frac{1}{a} \log \frac{1}{(a-\theta)} + \log \frac{1}{(a-\theta)} \log x +\log^2 x \right)\le \frac{ x^\alpha}{(a-\theta)a} \log^2 \frac{x+1}{(a-\theta)}.
\end{equation}
The implied constant depends only on $A$ and $\kappa$.
\end{lemma}
\begin{proof} Referring to Lemma \ref{l:Littlewood} and a small partial integration furnish
\begin{align*}
\sum_{\rho \in \Z} \frac{x^\beta}{|\rho|}&  \le
\frac{x^\alpha}{a} N(a,5)+ x^\alpha \int_5^x \frac{1}{t} d N(a,t) = \frac{x^\alpha N(a,5)}{a}  + x^\alpha \left\{ \left[ \frac{1}{t} N(a,t)\right]_5^x +\int_{5}^{x} \frac{N(a,t)}{t^2}  dt \right\}
\\& \le x^\alpha \left\{\frac{N(a,5)}{a}+ \frac{N(a,x)}{x} +\int_{5}^{x} \frac{N(a,t)}{t^2} dt \right\}
\\& \ll x^\alpha \left\{ \frac{\log \frac{1}{(a-\theta)}}{a(a-\theta)} + \frac{1}{a-\theta} \left( \left(\log x + \log\frac{1}{a-\theta}\right) +\int_{5}^{x} \frac{\log t + \log\frac{1}{a-\theta}}{t} dt \right)\right\}.
\end{align*}
\end{proof}
\begin{corollary}\label{cor:Deltatheta} Assume $\theta<\theta_0<1$, and let $x\ge 1$ be arbitrary. Then we have
\begin{equation}\label{eq:Deltaiftheta0}
|\D(x)| \ll \frac{1-\theta}{(\theta_0-\theta)^3} \log^3\frac{x+1}{\theta_0-\theta} ~x^{\frac{\theta+2\theta_0}{3}}~+ \frac{1}{(\theta_0-\theta)\theta_0} \log^2 \frac{x+1}{\theta_0-\theta} ~x^{\theta_0}
\end{equation}
The implied constant depends only on $A$ and $\kappa$.
\end{corollary}
\begin{proof}
We apply Proposition \ref{prop:vonMangoldt} with $b:=\frac{2\theta_0+\theta}{3}$ and picking some $t_k\in [x,x+5]$. (By construction of $\Gamma$ in Lemma \ref{l:path-translates}, such a $t_k$ exists.) This furnishes
\begin{equation}\label{eq:Deltaiftheta0first}
|\D(x)| \ll \left|\sum_{\rho \in \Z(\Gamma,t_k)}
\frac{x^{\rho}}{\rho}\right| + O\left( \frac{1-\theta}{(\theta_0-\theta)^3}\log^3\frac{x+1}{\theta_0-\theta} ~x^b \right).
\end{equation}
Here the $O$ term is the first expression on the right hand side of the asserted final estimate, so the proof hinges upon the estimation of the sum over the zeroes in $\Z(\Gamma,t_k)$. Recall that by construction of $\Gamma$ with $a:=\frac{b+\theta}{2}=\frac{2\theta+\theta_0}{3}$ these zeroes form a subset of those in the rectangle $[a,1]\times [-it_k,it_k]$. Given that there are no zeroes with $\Re \rho >\theta_0$, actually $\Z(\Gamma,t_k)\subset [\frac{2\theta+\theta_0}{3},\theta_0]\times [-i(x+5),i(x+5)]$. Therefore, for the sum over the zeroes belonging to $\Z(\Gamma,t_k)$ we can apply the above Lemma \ref{l:ZerosetZ} leading to the proof of the assertion.
\end{proof}

\begin{corollary} Let $\theta<\theta_0\le 1$ and let $\theta<\alpha<\theta_0$ and $x\ge 1$ be arbitrary. Then we have
\begin{equation}\label{eq:DSZtheta0}
D(x) \le S(x) \le Z_\alpha(x+5) +  O\left(\frac{1-\theta}{(\alpha-\theta)^3} \log^3\frac{x+1}{\alpha-\theta} ~x^\alpha \right).
\end{equation}
Moreover, we also have
\begin{equation}\label{eq:Zalphatheta0}
Z_\alpha(x) \ll \frac{1}{(\alpha-\theta)\alpha} \log^2 \frac{x+1}{\alpha-\theta} ~x^{\theta_0}.
\end{equation}
The implied constants depend only on $A$ and $\kappa$.
\end{corollary}
\begin{proof} The first inequality of \eqref{eq:DSZtheta0} is trivial. For the second it suffices to prove
$ |\D(x)| \le Z_\alpha(x+5) + O\left(\frac{1-\theta}{(\alpha-\theta)^3} \log^3\frac{x+1}{\alpha-\theta} ~x^\alpha\right)$, because the functions appearing here on the right hand side are monotonically increasing. This is completely analogous to \eqref{eq:Deltaiftheta0first}, with the only difference that we choose $b=\alpha$ and will therefore encounter a sum over zeroes lying in $\Z(\Gamma,t_k)\subset [\frac{\alpha+\theta}{2},\theta_0]\times [-i(x+5),i(x+5)]$, while the error term is the expression on the far right of \eqref{eq:DSZtheta0}. Now, $\Z(\Gamma,t_k)$ contains all zeroes in $\Z(\alpha,t_k)$, plus a subset of zeroes from $\Z(\frac{\alpha+\theta}{2},\alpha;t_k):=\Z(\frac{\alpha+\theta}{2};t_k)  \setminus \Z(\alpha;t_k)$, which is the set of all zeroes in $[\frac{\alpha+\theta}{2},\alpha]\times[-it_k,it_k]$. So, for our set $\Z^*:=\Z(\Gamma,t_k)\setminus \Z(\alpha;t_k) \subset [\frac{\alpha+\theta}{2},\alpha]\times[-it_k,it_k]$ Lemma \ref{l:ZerosetZ} applies with $a=\frac{\alpha+\theta}{2}$, and we get for its contribution
$$
\sum_{\rho \in \Z^*} \frac{t_k^\beta}{|\rho|} \ll  \frac{1}{(\alpha-\theta)\alpha} \log^2 \frac{t_k+1}{\alpha-\theta} ~t_k^{\alpha} \ll  \frac{1}{(\alpha-\theta)\alpha} \log^2 \frac{x+1}{\alpha-\theta} ~x^{\alpha}.
$$
The estimation of $Z_\alpha(x)$ is even easier using the same Lemma \ref{l:ZerosetZ}.
\end{proof}
 Collecting the above results, we can summarize the state of the matter as follows.
 \begin{theorem}\label{th:thetalessoneupper} Assume $\theta\le \theta_0\le 1$. Then we trivially have $|\Delta(x)|, D(x)\le S(x)$, while for $S(x)$ it holds
\begin{equation}\label{eq:thetalessoneupper}
S(x) \le \begin{cases} x^\theta \log^6(x+1) &\textrm{if}\quad \theta=\theta_0; \\
Z_\alpha(x) +  O\left(\frac{1-\theta}{(\alpha-\theta)^3} \log^3\frac{x+1}{\alpha-\theta} ~x^\alpha \right) \ll
\frac{1-\theta}{(\alpha-\theta)^3} \log^3\frac{x+1}{\alpha-\theta} ~x^{\theta_0} &\textrm{if}\quad \theta<\alpha<\theta_0.
\end{cases}
\end{equation}
\end{theorem}

\begin{corollary}\label{cor:thetanulllessone} Assume $\theta<\theta_0<1$. Then we have
$$
|\Delta(x)|, D(x)\le S(x) \le Z_{\frac{\theta+\theta_0}{2}}(x)+O\left(\frac{1}{(\theta_0-\theta)^3}\log^3\frac{x+1}{\theta_0-\theta} ~x^{\frac{\theta_0+\theta}{2}}\right) \ll x\exp(-(1-\ve)\om(x)).
$$
\end{corollary}

The above estimates are satisfactory if $\theta_0<1$. However, if $\theta_0=1$, \eqref{eq:Zalphatheta0} and \eqref{eq:thetalessoneupper} provide only a weak estimate, even weaker than the straightforward direct estimation $|\Delta(x)|\ll x\log x$. More importantly, to obtain an $\ll x~\exp(-(1-\ve)\om(x))=W(x) \left({x}/{W(x)}\right)^\ve$ upper estimate, in case $\theta_0<1$ the above bounds are more than sufficient given that $\om(x)\sim (1-\theta_0)\log x$ (and in fact $\om(x)>(1-\theta_0)\log x$), so that the allowed extra error term $(x/W(x))^{\ve}$ provides a room as large as a small power\footnote{In other words, consider that $\log(x\exp(-(1-\ve)\om(x)))\sim (1-(1-\ve)(1-\theta_0))\log x=(\theta_0+\ve(1-\theta_0))\log x$.} of $x$. But if $\theta_0=1$, $\om(x)$ can be of much smaller order of magnitude, and the estimates need to be much finer. E.g. in the Diamond-Montgomery-Vorhauer paper \cite{DMV} there was constructed a zero distribution where $\omega(x)$ is as low as $\sqrt{\log x}$, a much smaller exponent which cannot be controlled so easily.

To improve on this, we need to invoke the Carlson-type density theorem, too. So let us turn to the somewhat more delicate case when $\theta_0=1$. Recall that $\Re s=1$ is free of zeroes, thus if $\theta_0=1$, then we must have an infinitude of zeroes with $\rho_k=\beta_k+i\gamma_k$ and $\beta_k\to 1$, resulting in $\omega(x)\le \min_k (1-\beta_k)\log x+ \log \gamma_k =o(\log x)~(x\to \infty)$. Nevertheless, the well-known zero-free region $\si >1-c/\log (|t|+2)$ implies that for any zero $1-\beta>c/\log(|\gamma|+2)$, whence we find $\omega(x)\ge 2\sqrt{(1-\beta)\log x \log(|\gamma|+2)} \gg \sqrt{\log x}\to \infty$.

\begin{proposition}\label{p:Zomega} Assume $\theta_0=1$. If $a>\theta$ and $\ve>0$ then for sufficiently large $x>x_0(\G,\ve)$ we have
\begin{equation}\label{eq:Zsumtheta01}
Z_a(x)\ll_{a,\ve,A,\kappa,\theta} x\exp(-(1-\ve)\ot(x)).
\end{equation}
\end{proposition}
\begin{proof}
Put $\vp:=\ve/K$, where $K$ is a large constant to be chosen later, and put also
$\alpha:=1-4\vp$ and $\alpha_0:=\max\{\beta:=\Re \rho~:~ \ze(\rho)=0, |\gamma|=|\Im \rho|\le 5\}$, pointing out that $\alpha_0<1$, too. We can assume that $\ve$ is so small that $\alpha_0<\alpha$. We write similarly as above,
\begin{align*}
Z_a(x)&= \sum_{\rho \in \Z(a,x)\setminus\Z(\alpha,x)} \frac{x^{\beta}}{|\rho|} + Z_\alpha(x)
\ll_{A,\kappa,\theta} \frac{ x^{\alpha}}{(a-\theta)a} \log^2 \frac{x+1}{a-\theta}+ \frac{N(\alpha,5)}{\alpha} x^{\alpha_0} + \sum_{n=1}^{\infty} \sum_{\rho\in \Z(\alpha; e^n,e^{n+1})}   \frac{x^{\beta}}{|\rho|}
\\ &\ll_{A,\kappa,\theta,a,\vp} x^{\alpha+\vp} +\sum_{n=1}^{\infty} N(\alpha;e^n,e^{n+1}) \max_{\rho=\beta+i\gamma \in \Z(\alpha; e^n,e^{n+1})}\frac{x}{\exp((1-\beta)\log x +\log|\rho|)}
\\& \ll_{A,\kappa,\theta,a,\vp} x^{\alpha+\vp} + \sum_{n=1}^{\infty} N(\alpha;e^{n+1})
x\exp(-(\inf_{\rho\in \Z(\alpha; e^n,e^{n+1})} (1-\beta)\log x +\log|\rho|)).
\end{align*}
Here the first term gives only an $O(x^{1-\vp})$ error term. For the sum over zeroes close to $\Re s=1$, we appeal to the Carlson-type density estimate Theorem \ref{th:NewDensity} with our fixed $\vp>0$. This leads to $N(\alpha;e^{n+1}) \ll_{A,\kappa,\theta,\vp} \exp((\frac{12}{1-\theta}(1-\alpha)+\vp)(n+1))$. Also, for the inf inside the sum first we can separate a part and then the inf can be extended to all zeroes:
\begin{align}\label{eq:Knvp}
\inf_{\rho\in \Z(\alpha,e^n,e^{n+1})} &(1-\beta)\log x +\log\gamma \notag
\\& \ge  K\vp \inf_{\rho\in \Z(\alpha,e^n,e^{n+1})} \log \gamma ~+~ (1-K\vp) \inf_{\rho \in \Z(\alpha,e^n,e^{n+1})} (1-\beta)\log x +\log\gamma \notag
\\& \ge K\vp n + (1-K\vp)\ot(x).
\end{align}
In all we find with choosing $K:= \frac{48}{1-\theta} +2$ the estimate
\begin{align*}
Z_a(x) & \ll_{A,\kappa,\theta,a,\vp}  x^{1-\vp} + \sum_{n=1}^\infty N(\alpha,e^{n+1}) x \exp(-K\vp n -(1-K\vp)\ot(x))
\\ & \ll_{A,\kappa,\theta,a,\vp} x^{1-\vp} +  \sum_{n=1}^\infty  \exp\left((\frac{12}{1-\theta}\,4 \vp+\vp)(n+1)-K\vp n \right) x \exp(-(1-\ve)\ot(x))
\\ & \le x^{1-\vp} + \sum_{n=1}^\infty  e^{K\vp} \exp\left(-\vp n \right) x \exp(-(1-\ve)\ot(x))
\\ & \ll_{A,\kappa,\theta,a,\vp} x^{1-\vp} + x \exp(-(1-\ve)\ot(x)) \ll x\exp(-(1-\ve)\ot(x)) \qquad(x>x_0(\G,\ve)).
\end{align*}
\end{proof}

\begin{theorem}\label{th:upperestimation} If $\theta_0=1$, then we still have for any $\ve>0$ and sufficiently large $x>x_0(\ve,\G)$ the inequalities
\begin{equation}\label{eq:uppertheorem}
D(x) \le S(x) \le (1+o{_\G}(1)) Z_\alpha(x) \ll_{\ve,\G} x\exp(-(1-\ve)\om(x)).
\end{equation}
\end{theorem}

\begin{corollary}\label{cor:ometaomupper} If $\eta:[0,\infty)\to [0,1/2]$ is a function such that the domain \eqref{eq:etazerofree} is free of zeroes of the Beurling $\zeta$ function, then we have $D(x) \le S(x) \ll x\exp(-(1-\ve)\om_\eta(x))$.
\end{corollary}
\begin{proof} Observe that for any fixed particular root $\rho$ of $\zeta$ we have $\om(\rho,x)=(1-\beta)\log x+ \log \gamma \ge \eta(\gamma) \log x + \log \gamma \ge \om_\eta(x)$ as the latter is a minimum of such expressions with $\gamma$ allowed to run over all values of $t>1$, irrespective of being a concrete ordinate of a zero $\rho$ or not. Therefore, the same holds also for the minimum taken over all zeroes, i.e. $\om(x)\ge \om_\eta(x)$. Consequently, the assertion follows from Theorem \ref{th:upperestimation} if $\theta_0=1$ and from Proposition \ref{p:thetathetanull} and Corollary \ref{cor:thetanulllessone} for $\theta_0=\theta$ or $\theta<\theta_0<1$, respectively.
\end{proof}

Here it becomes clear why the Ingham type "direct" theorem of Pintz formulated with the function $\om$ is stronger than any result with a zero-free domain defined by some boundary function $\eta$.

Also note that Corollary \ref{cor:ometaomupper} implies the above stated Theorem \ref{th:Beurlingdomainesti} \emph{with arbitrary boundary functions} $\eta$, assuming neither continuity nor monotonicity.

This is in particular useful if we consider that there were studies related to such boundary functions. Namely, in the classical case it was studied in detail, see e.g. \cite{Trud, Trud2}, what consequences can be drawn from knowing RH to hold \emph{up to a certain given height} (which is already a known fact for quite high ordinates in some cases, most notably for the Riemann zeta function \cite{Trud3}). Now, such a condition could be described by setting $\eta(t):=1/2$ for $0\le t \le T$ and then just putting $\eta(t)=1$, or perhaps $\eta(t)=c/\log t$ or anything known for a zero-free domain: in Corollary \ref{cor:ometaomupper} we do not need continuity or monotonicity.

Also we can discuss here what the result of Theorem \ref{th:upperestimation} means. It basically says that from the series in the Riemann-von Mangoldt formula \eqref{eq:Riemann-vonMangoldt}, however divergent, we can estimate the order of magnitude of $\D(x)$. It is clear in this direction that the total sum of absolute values of terms i.e. $Z(x)$ should bound $\D(x)$--and whence also $D(x)$ and $S(x)$--from above, but it is less immediate that this upper bound is still below (almost) $x\exp(-\om(x))$. The meaning of this last expression is, however, clear: that stands for the largest term of the Riemann-von Mangoldt sum. Therefore, the results show that in fact this Riemann-von Mangoldt sum--in spite of divergence!--behave quite regularly in the sense that the order of the total sum is about the magnitude of the largest term of the series. That is a well-known behavior in e.g. entire functions--but it holds also for this divergent series, which might be somewhat surprising. The analysis of the "converse Ingham type theorems" to follow below will show, however, that the strong connection between largest term and total sum will prevail also in lower estimations, that is $\D(x)$ will be shown to oscillate frequently about as large as the total sum (which is already known to be bounded by about the largest term of the series).

Obviously $\D(x)$ cannot \emph{always} be as large, given that we expect it changing signs often, therefore being small around such sign changes; yet the finding that \emph{frequently it reaches about} $x\exp(-\om(x))$, while the total absolute value sum $Z_a(x)$ of the Riemann-von Mangoldt series is \emph{always at most that magnitude}, is remarkable.

Recall that all these were explored by Pintz \cite{PintzProcStekl} in the classical case and here we only succeeded to extend some of his results to the Beurling case.

\section{Oscillation of $\Delta_{\PP}(x)$ in case $\theta_0<1$}\label{sec:oscillationthetanull}

Having the study of $O$-results, i.e., upper estimations concluded, in the rest of the paper we study oscillation estimates, i.e., $\Omega$-type results. As is well understood in the classical case and is extending with ease to the Beurling setup, the case when $\theta_0<1$ is much simpler, being essentially a century-old result of Phragmen. For clarification of the full picture, we recall here a version of the well-known argument.

Let us start with observing that the average
$$
A(x):=\frac{1}{x} \int_1^x \D(u) du
$$
is at most $D(x)$, which is estimated from above by $S(x)$. Further, if $\D(x) =O(x^a)$ with a certain exponent $\theta \le a \le 1$, then obviously also $A(x) = O(x^a)$.

The following assertion essentially settles the case of $\theta <\theta_0<1$.

\begin{proposition}\label{p:thetanulllessthanone} If $\theta<\theta_0 \le 1$, then for any $\ve>0$ we necessarily have $A(x)=\Omega(x^{\theta_0-\ve})$.
\end{proposition}
\begin{proof} Assume the contrary, that is, assume $A(x)=O(x^{\theta_0-\ve})$. Let us consider the Dirichlet-Mellin transform ${\mathcal M}(A)(s):=\int_1^\infty x^{-s} dA(x)$ of $A(x)$, which--under our indirect assumption--converges uniformly and hence is \emph{necessarily analytic} for $\Re s>\theta_0-\ve$. According to the upper part of the Chebyshev bounds (following from the PNT and hence in turn from Beurling's original result in case of Axiom A holding), we always have\footnote{Actually, we have proved much better in Corollary \ref{cor:Deltatheta}, which allows to treat these integral reformulations everywhere in $\Re s>\theta_0$. However, referring to the Chebyshev bound only is self-contained, not relying on previous calculations and the Carlson-type density theorem in particular. Using the analytic continuation is needed anyway, either to $\Re s \in (\theta_0-\ve,1]$ or only to $(\theta_0-\ve,\theta_0]$, which is all the same.} $\D(x)=O(x)$. Therefore, the following reformulations are valid at least for $\Re s >1$, treating only locally uniformly convergent integrals:
\begin{align*}
{\mathcal M}(A)(s) &:= \int_1^\infty x^{-s} dA(x) = - \int_1^\infty A(x) dx^{-s}= -\int_1^\infty \left( \frac{1}{x} \int_1^x \D(u) du \right) dx^{-s}
\\ & = -\int_1^\infty \int_1^\infty \D(u) \int_u^\infty \frac{1}{x} dx^{-s}du = \int_1^\infty \int_1^\infty \D(u) \frac{s}{s+1} u^{-s-1}du
\\& = \frac{-1}{s+1} \int_1^\infty \D(u) d u^{-s} = \frac{-1}{s+1} \left(1-\int_1^\infty u^{-s}d\D(u) \right)= \frac{1}{s+1}\bigg(-1+{\mathcal M} (\D)(s)\bigg) .
\end{align*}
Here the Dirichlet-Mellin transform of $\D(x)$ can be calculated directly using \eqref{zetalogder}, \eqref{psidef} and \eqref{Deltadef}:
\begin{equation}\label{Hdef}
H(s):={\mathcal M}(\D)(s) := \int_1^\infty x^{-s} d\Delta(x) = \int_{1}^{\infty} x^{-s} d\psi(x) - \int_{1}^{\infty} x^{-s}dx = -\frac{\zeta'}{\zeta}(s)- \frac{1}{s-1}.
\end{equation}
Thus ${\mathcal M}(A)(s) = \frac{-1}{s+1}  \left(\frac{\zeta'}{\zeta}(s)+\frac{s}{s-1}\right)$, which is indeed meromorphic, \emph{but not analytic} in the whole halfplane $\Re s>\theta_0-\ve$, due to the poles arising from zeroes lying arbitrarily close to the $\Re s=\theta_0$ line.

The obtained contradiction proves the assertion.
\end{proof}

\begin{corollary}\label{cor:thetanullosci} If $\theta<\theta_0<1$, then we have $\D(x), D(x), S(x) =\Omega(x\exp(-(1+\ve)\omega(x))$.
\end{corollary}
\begin{proof} Note that $\omega(x)\sim (1-\theta_0) \log x$, which is a positive constant multiple of $\log x$ in our case. By the previous Proposition \ref{p:thetanulllessthanone} the assertion follows.
\end{proof}

\begin{remark} The above result is ineffective. To get an effective version is certainly possible by applying e.g. a simplified form of the below machinery, worked out in the next three sections for the $\theta_0=1$ case. We leave the details to the reader.
\end{remark}

\begin{remark}\label{rem:nozero} The case when $\theta_0=\theta$, remains somewhat unclarified. This case means that there are no zeroes in $\Re s>\theta$, that is, there are no zeroes at all in the halfplane of analyticity (meromorphicity) guaranteed by Axiom A--and hence by definition $\theta_0=\max(\theta,\sup_\emptyset\beta)=\max(\theta,-\infty)=\theta$.

In this case various scenarios are indeed possible, hence it is no wonder that we cannot prove an $\Omega(x\exp(-(1+\ve)\om(x))$ result. E.g. there is a construction \cite{BrouckeDebruyneVindas} of a Beurling system with very well-behaved primes satisfying $\D(x)=O(x^{1/2+\ve})$, but the integers satisfying Axiom A with no $\theta<1$. Analogously, a similar construction with any "best value" of $\theta$ in $(1/2,1)$ seems likely to exist. We also consider that with some refinement of the argument in \cite{Rev-One} (see in particular Remark 1 there), even $\D(x)=O(x^\alpha)$ is possible with any "best value\footnote{Formally, this "best value" is $\limsup_{x\to \infty} \frac{\log(\N(x)-x)}{\log x}$.}" of $\theta \in [1/2,1]$ in Axiom A. That, on the other hand, a situation with $\theta_0=\theta<1/2$ cannot occur, follows from a nice result of Hilberdink \cite{H-10}, taking into account the above Theorem \ref{th:thetalessoneupper}, say.

Describing the above mentioned constructions would lead us aside in the direction of studying so-called "$\alpha$-$\beta$-systems" of Hilberdink, and is therefore left to a forthcoming study.
\end{remark}

\section{The weighted average of $\Delta_{\PP}(x)$ and its upper estimate}\label{sec:upperesti}

As told, proving an oscillatory result of the magnitude $x\exp(-(1+\ve)\omega(x))$ is much harder when $\theta_0=1$, and $\omega(x)=o(\log x)$. The other cases being already discussed, from now on we consider only the case $\theta_0=1$. To prove our oscillatory results in this case will comprise three sections, the present one being devoted to a more or less direct upper estimation of a weighted and averaged version of $\Delta(x)$. The next section will apply a different (complex) calculus to demonstrate also a lower bound, and by a comparison of the two sided estimates and with suitable choices of our parameters we will conclude the argument in Section \ref{sec:parametersandproof}. So, these sections are parts of the same argument and our notations, choices, conditions and constructions remain valid throughout.

In this argument we will assume that $x$ is large enough, and that $\rho_0=\beta_0+i\gamma_0$ is a $\zeta$-zero, which is close to optimal for $\omega(x)$: more precisely\footnote{Recall that here we assume that $\theta_0=1$, whence we have that it is not attained. Therefore, when $x\to \infty$, also for the corresponding extremal zero we have $\rho \to 1+i\infty$. Therefore, from here on we do not distinguish between writing $\om(x)$ with $\log|\rho|$ or $\ot(x)$ with $\log \ga$. The difference is only $\log(|\rho|/\ga)=\frac12 \log(1+\beta^2/\ga^2)< \frac12 \ga^{-2}$, a negligible term in all estimates.} $\omega(\rho_0;x):=(1-\beta_0)\log x + \log \gamma_0 < \omega(x)+1$, an unimportant constant deficiency. This flexibility in possibly preferring some close-by zero instead of the actual extremal one we need here for ensuring another, somewhat more convenient property: we assume that $(\beta_0,1]\times i[\gamma_0-3,\gamma_0+3]$ is zero-free, so that $\rho_0$ is a kind of locally borderline zero. To see that our assumptions can be met, consider a truly extremal zero $\rho_1$. If $[\beta_1,1]\times i[\gamma_1-3,\gamma_1+3]$ is zero-free, then we are done: $\rho_0:=\rho_1$. If not, then we find another zero with $\beta_2>\beta_1$ and $\gamma_1<\gamma_2\le \gamma_1+3$. (Note that $\gamma_2<\gamma_1$ is impossible, for then $\om(\rho_2,x)<\om(\rho_1,x)$, contradicting to the definition of $\rho_1$. However, in the forthcoming later steps it can happen that the values of $\ga_j$ are not strictly increasing, even if $\gamma_j>\ga_0$, always.) We keep continuing the process: either the process ends before $k$ steps, or we have $\rho_k=\beta_k+i\gamma_k$, $\beta_k>\beta_{k-1}>\dots>\beta_1$, $\gamma_k<\gamma_1+3k$. However, the Carlson Density Theorem says that the number of zeroes $N(1-\ve,T)=o(T)$, whence for a small, fixed $\ve>0$, and large enough $x$ so that $\gamma_1>T_0(\ve)$, we obtain that the above process cannot produce more than $N(1-\ve,\gamma_1+3k)=o(\gamma_1+3k)$ zeroes. Therefore $k=o(\gamma_1+3k)$ and thus $k=o(\gamma_1)$ and $\log \gamma_1 < \log \gamma_k < \log (2\gamma_1)<\log\gamma_1+\log 2$, as needed. Then indeed $\om(\rho_k,x)=(1-\beta_k)\log x + \log \gamma_k < (1-\beta_1)\log x + \log \gamma_1 +1 =\om(x)+1$, so that the terminal element $\rho_k$ of the above construction is good for $\rho_0$ with the claimed properties.

In the following we will often write $\omega:=\om(\rho_0;x)$, and use without repeated explanations the fact that $0\le \om(x)-\om<1$, so that we can handle $\om$ as the extremal value.

We use the Dirichlet-Mellin transform \eqref{Hdef} of $\D(x)$, which is meromorphic in $\Re s>\theta$ and admits the further reformulations\footnote{As is mentioned above in Section \ref{sec:oscillationthetanull}, all these integral representations are in fact locally uniformly convergent even in $\Re s>\theta_0$ in view of Theorem \ref{th:thetalessoneupper}.}
\begin{equation}\label{eq:Rdef}
H(s):=\int_1^\infty x^{-s} d\Delta(x) = 1-\int_1^\infty \D(x)dx^{-s} = 1-s \int_{1}^{\infty} \D(x)x^{-s-1} dx.
\end{equation}
Further, we define with certain constant parameters $L, M$ to be specified later
\begin{align}\label{eq:U} \notag
U:=U(\rho_0)& := \frac{1}{2\pi i} \int_{2-i\infty}^{2+i\infty} H(s+\rho_0) e^{Ls^2+Ms}ds
 \\ &= \frac{1}{2\pi i} \int_{2-i\infty}^{2+i\infty} \left(1 -\int_1^{\infty}
\D(x) \frac{d}{dx} (x^{-s-\rho_0}) dx \right) e^{Ls^2+Ms} ds \notag
\\&= \frac{1}{2\sqrt{\pi L}}\exp(-M^2/4L) - \int_1^{\infty} \D(x) \frac{d}{dx} \left\{ x^{-\rho_0}
\frac{1}{2\pi i} \int_{(2)} e^{Ls^2+(M-\log x)s}ds\right\}dx
\\ &= \frac{e^{-M^2/4L}}{2\sqrt{\pi L}} - \int_1^{\infty} \D(x) \frac{d}{dx} \left\{ x^{-\rho_0}\frac{1}{2\sqrt{\pi L}} \exp \left( -\frac{(\log x -M)^2}{4L}
\right)\right\} dx \notag
\\ &= \frac{e^{-M^2/4L}}{2\sqrt{\pi L}}+\frac{1}{2\sqrt{\pi L}}
\int_1^{\infty} \frac{\D(x) }{x}  x^{-\rho_0} \left\{\frac{\log x
-M}{2L} +\rho_0\right\} \exp \left( -\frac{(\log x -M)^2}{4L} \right)
dx, \notag
\end{align}
where the integral formula of Lemma \ref{l::integralformula}
was applied twice, and the order of the integration and the derivation was
changed.

For the following we set a few parameters, whose values will be specified more precisely only later, but here we already tell about their approximate order of magnitude. Namely, we set $\ell \in (0,1)$ a fixed constant and $0<\vp<\ve$ a sufficiently small number, satisfying $\vp<\sqrt{(1-\theta)/300}$). Further, we introduce with these set quantities also the parameters $m, M$ and $L$ satisfying
\begin{equation}\label{eq:mMLdef}
\vp \omega  \le m \le \vp \log x, \qquad \log x -2m \le M \le \log x -m, \qquad L:=\ell M.
\end{equation}

Now let us split the integral for $U$ to four parts as
\begin{equation}\label{eq:Usplit}
U_0:=\int_1^{\sqrt{x}},~ U_1:=\int_{\sqrt{x}}^q,~U_2:=\int_q^x,~U_3:=\int_x^{\infty} \quad \textrm{with} \quad q:=\exp(M-m).
\end{equation}
Note that here the interval $[q,x]$ of integration for $U_2$ is a subset of $[x\exp(-3m),x]$.

For the general estimation of $\psi (x)$ and $\Delta(x)$, we have the obvious estimates
\begin{align}\label{eq:Psirough}
(0\le ) ~\psi(x) &\le \sum_{g\in \G, |g|\le x} \log x = \N(x) \log x \le (A+\kappa) x \log x \qquad &(x\ge 1), \notag
\\ |\Delta(x)|&\le A_0 x \log x +1 \qquad \textrm{where} \quad A_0:=\max(1,A+\kappa) \qquad &(x\ge 1).
\end{align}
This we will use for $U_0$ only, while for the most part we want to estimate more finely, so that we will apply the above Theorem \ref{th:upperestimation} with $\vp$ in place of $\ve$ and assuming $\sqrt{x}>x_0(\G,\vp)$.

For the interval $[x,\infty)$ we have $\om(u)>\om(x)$ by monotonicity, whence by Theorem \ref{th:Beurlingdomainesti}
$$
\frac{|\D(u)|}{u^{1+\beta_0}} \ll \frac{u\exp(-(1-\vp)\om(u))}{u^{1+\beta_0}} \le \exp(-(1-\vp)\om(x)) u^{1-\beta_0} \frac{1}{u} \qquad (u>x).
$$
Thus with the notation $Q:=\exp(M+m)<x$ the quantity $U_3$ can be estimated as follows.
\begin{align}\label{eq:U3b}
|U_3| & \ll \frac{\exp(-(1-\vp)\om(x))}{\sqrt{\pi L}} \int_x^{\infty} u^{1-\beta_0} \exp
\left( -\left(\frac{(\log u -M)}{2\sqrt{L}}\right)^2 \right) \left(\frac{\log u-M}{2L}+|\rho_0|\right)\frac{du}{u}\notag
\\ &\ll \frac{e^{-(1-\vp)\om}x^{1-\beta_0}|\rho_0|}{\sqrt{\pi L}} \int_{x}^{\infty} \exp \left( (\log u-\log x)(1-\beta_0) -\left(\frac{(\log u -M)}{2\sqrt{L}}\right)^2 \right) \left(\frac{\log u-M}{2L|\rho_0|}+1\right)\frac{du}{u}\notag
\\ &\ll\frac{e^{\vp\om}}{\sqrt{L}} \int_{Q}^{\infty} \exp \left( (\log u-M)(1-\beta_0) -\left(\frac{(\log u -M)}{2\sqrt{L}}\right)^2 \right) \left(\frac{\log u-M}{2L|\rho_0|}+1\right)\frac{du}{u}\notag
\\ &=2e^{\vp\om} \int_{\frac{m}{2\sqrt{L}}}^{\infty} \exp \left((1-\beta_0)2\sqrt{L}y -y^2 \right) \left(\frac{y}{\sqrt{L}|\rho_0|}+1\right)dy
\\ &\ll e^{\vp\om} \int_{\frac{m}{2\sqrt{L}}}^{\infty} \exp \left(\frac{\om}{\log x} 2\sqrt{L}y -y^2 \right) \left(2y-\frac{\om}{\log x} 2\sqrt{L}\right)dy \le e^{\vp\om+\om m/\log x -m^2/4L} .\notag
\end{align}
Here in the last but one step we wrote in the obvious $(1-\beta_0)\le \om/\log x$ and estimated $\frac{y}{\sqrt{L}|\rho_0|}+1 \le 2y-\frac{\om}{\log x} 2\sqrt{L}$, which certainly holds if  {$m>4\sqrt{L}$ and $m>8 \om L/\log x $} hold simultaneously, given that $y\ge \frac{m}{2\sqrt{L}}$ and $|\rho_0|>1/\sqrt{L}$.

\medskip
Let us see the analogous case of the estimation of $U_1$. Here we use the monotonicity of $\om(u)$, the inequalities $\omega(u)\le \om(\rho_0;u)\le \om(\rho_0;x)=\om \le \om(x)+1$ and also the end estimate \eqref{omegayz} from Lemma \ref{l:omega} in the form $\om(x)-\om(u)\le \frac{\om(u)}{\log u} (\log x - \log u)$. As now $\sqrt{x}\le u \le q$, we can write $\log x -M \le 2m = 2(M-\log q) \le 2 (M-\log u)$, whence also $\log x - \log u \le 3(M-\log u)$. It follows that $\frac{\om(u)}{\log u} (\log x - \log u) \le \frac{\om}{\frac12 \log x} (\log x -\log u) \le 6\frac{\om}{\log x} (M-\log u)$. Thus an application of Theorem \ref{th:upperestimation} (with $\vp$ and correspondingly sufficiently large $x$) yields
\begin{align*}
\frac{|\D(u)|}{u^{\beta_0}} |\rho_0| & \le \exp(-(1-\vp)\om(u)) u^{1-\beta_0} |\rho_0| =
e^{\vp\om(u)} \exp(\om(\rho_0,u)-\om(u))
\\ & \le e^{\vp\om} \exp(\om(\rho_0,x)-\om(u))\le e^{\vp\om} \exp(\om(x)+1-\om(u))
\\ & \le e^{1+\vp\om} \exp\left(\frac{\om(u)}{\log u}(\log x-\log u)\right)\le 3 e^{\vp\om} \exp\left(6 \frac{\om}{\log x}(M-\log u)\right).
\end{align*}

Using this and recalling $q=e^{M-m}$ we infer the estimate
\begin{align}\label{eq:U1}
|U_1|&\leq \frac{3e^{\vp\om}}{2\sqrt{\pi L}} \int_{\sqrt{x}}^{q} \exp \left( 6 \frac{\om}{\log x}(M-\log u)  -\frac{(\log u -M)^2}{4L} \right) \left\{\frac{M-\log u}{2L|\rho_0|} +1\right\}  \frac{du}{u} \notag
\\ & \ll e^{\vp\om} \int_{m/2\sqrt{L}}^{\infty} \exp \left(  6 \frac{\om}{\log x} 2\sqrt{L}y -y^2 \right) \left\{\frac{y}{\sqrt{L}|\rho_0|} +1\right\}  dy
\\ & \leq e^{\vp\om} \int_{m/2\sqrt{L}}^{\infty} \exp \left( 12\sqrt{L}\frac{\om}{\log x}y -y^2 \right) \left\{2y- 12\sqrt{L}\frac{\om}{\log x}\right\}  dy
\notag\\& =\exp(\vp\om+6\om m/\log x -m^2/4L). \notag
\end{align}

Here we have substituted $y:=(M-\log u)/2\sqrt{L}$ and applied the similar to the above {conditions $|\rho_0|\ge 2/\sqrt{L}$, $m>4\sqrt{L}$ and $m> 24L\om/\log x$} to get $\frac{y}{\sqrt{L}|\rho_0|} +1\le 2y-12\sqrt{L}\om/\log x$.

For small values of $u$ we do not have the asymptotic estimate of Theorem \ref{th:upperestimation}, so for the interval $[1,\sqrt{x}]$ we simply use \eqref{eq:Psirough} together with the  {assumptions $\log x \ge 5L\ge 20 \sqrt{L}$ (the last part coming from assuming $\ell<0.2$ and $L\ge 16$) and $M\ge 0.95 \log x$} (entailed by $\vp<0.025$, as $M\ge \log x -2m \ge (1-2\vp)\log x$). These provide
\begin{align}\label{eq:U0}
|U_0| &\ll \frac{1}{\sqrt{L}} \int_1^{\sqrt{x}} u^{1-\beta_0} (\log u+1) |\rho_0| \left\{\frac{M-\log u}{2L|\rho_0|} +1\right\}  \exp \left(-\frac{(M-\log u)^2}{4L} \right) \frac{du}{u} \notag
\\&\le x^{1-\beta_0} |\rho_0| ~ \frac{1}{\sqrt{L}} \int_1^{\sqrt{x}} 3(M-\log u) ~\frac{M-\log u}{L} ~ \exp \left(-\frac{(M-\log u)^2}{4L} \right) \frac{du}{u}
\\ & \ll e^\om \int_{(M-\log\sqrt{x})/2\sqrt{L}}^\infty y^2 e^{-y^2} dy \le
e^\om \frac{\log x}{\sqrt{L}} \exp\left(-0.225^2 \log^2 x/L\right) \ll e^{\om-0.04 \log^2x/L},\notag
\end{align}
applying in the last line the condition $M\ge 0.95 \log x$ and the estimate of Lemma \ref{l:estimates} (ii) with $\lambda=2$ and $B=(M-\log\sqrt{x})/2\sqrt{L} \ge 0.225 \log x/\sqrt{L} \ge 4$, this last estimate following from the above assumption $\log x \ge 20\sqrt{L}$, and entailing also
$$
20 \le \log x/\sqrt{L} \le 0.05 \log^2x/L \le \exp\left( \frac{\log 20}{20}(0.05\log^2x/L)\right)\le \exp\left(0.0075\log^2x/L\right).
$$
Noting that $M>(1-2\vp)\log x$ implies $e^{-M^2/4L}<e^{\om-0.04 \log^2x/L}$, collecting \eqref{eq:U}, \eqref{eq:Usplit}, \eqref{eq:U3b}, \eqref{eq:U1} and \eqref{eq:U0} we are led to
\begin{equation}\label{S-upperesti}
|U| \le |U_2|+O\left(e^{\om-0.04 \log^2x/L}+e^{\vp\om+6\om m/\log x -m^2/4L}\right).
\end{equation}
Regarding $U_2$,
we can write
\begin{align}\label{eq:U2upperDelta}
|U_2(\rho_0)| & \le \frac{1}{2\sqrt{\pi L}} \int_q^x \frac{|\D(u)|}{u^{1+\beta_0}} \left(|\rho_0|+ \frac{|\log u - M|}{2L} \right)  \exp \left(-\frac{(M-\log u)^2}{4L} \right) du \notag \\& \le \left(1+ \frac{2m}{2L|\rho_0|} \right) \frac{1}{2\sqrt{\pi L}} \int_{M-m}^{\log x} \frac{|\D(e^v)| |\rho_0|}{e^{\beta_0v}}  \exp \left(-\frac{(M-v)^2}{4L} \right)  dv
\\ &\ll  \max_{\xi=e^v \in [q,x]} \frac{|\D(\xi)|}{\xi^{\beta_0}/|\rho_0|} \frac{1}{2\sqrt{\pi L}} \int_{M-m}^{\log x} \exp \left(-\frac{(M-v)^2}{4L} \right)  dv \le \max_{\xi=e^v \in [q,x]} \frac{|\D(\xi)|}{\xi^{\beta_0}/|\rho_0|},\notag
\end{align}
under the  {additional assumption that $m\ll |\rho_0|L$}.

These upper estimates will be compared to the lower estimation of the next section.

\section{Lower estimate of $U$ by contour integration and power sum theory}\label{sec:loweresti}

In this second part we calculate $U$ by using the first form in \eqref{eq:U}. More precisely, we are to transfer the line of integration of $U=U(\rho_0)$ with $\rho_0=\beta_0+i\gamma_0$ from $(\sigma=2)$ to a new contour $\Gamma-\beta_0$, with $\Gamma$ lying in the strip $1-3\de\le \Re s \le 1-\de$, where $\de>0$ is a small parameter, to be chosen later in such a way that $\de<0.01(1-\theta)$, whence in particular $\frac{1+\theta}{2}<1-3\de$, too. Recall that $\rho_0$ is a close-to-extremal zero for a large value of $x$, and $\theta_0=1$ implies that $\beta_0$ has to be arbitrarily close to $1$, whence for $x$ large enough we certainly have $\beta_0>1-\de$.

The integrand, $H(s+\rho_0)\exp(Ls^2+Ms)$, is meromorphic between $\Re s=2$ and $\Gamma-\beta_0$, and if the new contour avoids all singularities of $H$, then the Residue Theorem can be applied. To ascertain that the contour avoids zeroes of $\zs$ (poles of $\frac{\zeta'}{\zeta}(s)$) we recall that once a Beurling system satisfies Axiom A with a certain value of $\theta<1$, then it satisfies the same also with any other value $\theta'\in (\theta,1)$. Therefore, to construct $\Gamma$ we will apply the construction of Lemma \ref{l:path-translates} but with the translation set $\AAA:=\{-i\ga_0,0,i\ga_0\}$, with the Axiom A related parameter value $\theta':=1-3\de$, and with the additional parameter $b:=1-\de$. This will furnish a contour which indeed lies in the prescribed strip $a:=1-2\de\le \Re s \le b:=1-\de$, avoids all singularities of $H$, and satisfies the estimates\footnote{In fact, the constants in the estimations of these Lemmas will not much depend on $\de$ because a moment's thought reveals that in the proof of the related estimates in \cite{Rev-MP} we can still appeal to the known estimates for the number of $\ze$-zeroes, available in the previous lemmas up to the whole strip $\theta<\Re s\le 1$; so we still have these Lemmas with constants depending on $\theta$ and $\kappa$ only. NOT QUITE TRUE...} of Lemmas \ref{l:path-translates} and \ref{l:zzpongamma-c}.

The transition of the contour of integration can be done easily due to the estimates of Lemma \ref{l:zzpongamma-c} and the uniform bound $|e^{Ls^2+Ms}|=O_{L,M}(e^{-Lt^2})$ holding uniformly in the strip $-1\le \sigma \le 2$ and $s=\sigma+it$. By an application of the Residue Theorem we thus find after the change of the integration path
\begin{align}\label{Uotherway}
U(\rho_0) 
= \frac{1}{2\pi i} \int_{\Gamma-\beta_0} H(s+\rho_0) e^{Ls^2+Ms}ds + \sum^{\star}_{\rho} \exp\left(L(\rho-\rho_0)^2+M(\rho-\rho_0)\right),
\end{align}
where the $\star$ indicates that exactly those zeroes of the Beurling zeta function are taken into account (and then according to multiplicity) for which $\rho-\rho_0$ lie to the right of the new contour $\Gamma-\beta_0$, that is, for which $\rho-i\gamma_0$ is to the right of $\Gamma-i\gamma_0$. Recall that the singularities of $H(s+\rho_0)$ are exactly at translates $\rho-\rho_0$ of zeroes $\rho$ of $\zeta$ with residues according to multiplicity; and that by construction all translated $\zeta$-zeroes $\rho-\rho_0$ avoid points $s-\beta_0=\sigma-\beta_0+it$ of the curve $\Gamma-\beta_0$ --that is, all vertically translated zeroes $\rho-i\gamma_0$ avoid the points $s=\sigma+it$ of the curve $\Gamma$--by at least $d:=d(t):=d(b,\theta,n,\gamma_0;t)=d(1-\de,\theta,3,\gamma_0;t)$ given in \eqref{ddist-corr}.

Here the integral--which we will denote by $S_0=S_0(\rho)$ henceforth--can be estimated by Lemma \ref{l:zzpongamma-c} as follows.
\begin{align}\label{Newcontourint}
|S_0(\rho_0)|&=\bigg| \frac{1}{2\pi i} \int_{\Gamma-\beta_0}  H(s+\rho_0) e^{Ls^2+Ms}ds \bigg|
\\ & \le \int_{\Gamma-\beta_0} \frac{A_3}{\de^3} \left(\log (|t|+\gamma_0+5) + \log\frac{1}{\de}\right)^2 \exp\left(L((\beta_0-a)^2-t^2) +M(b-\beta_0)\right) |ds| \notag
\\ & \le\frac{A_4}{\de^5} ~ e^{L(\beta_0-a)^2+M(b-\beta_0)} \int_{\Gamma-\beta_0} \log^2 (|t|+\gamma_0+5) \exp(-Lt^2) |ds| \notag
\\ &  \le A_4 ~ e^{4L\de^2+M(1-\de-\beta_0)} \int_{\Gamma-\beta_0} \log^2 (|t|+\gamma_0+5) \exp(-Lt^2) |ds|. \notag
\end{align}

By construction, the broken line $\Gamma$ consists of horizontal line segments $H_k$ of length $\le \frac12(b-\theta')=\de$ at height $t_k$, and vertical segments the horizontal projection of which covers the imaginary axis exactly (apart from endpoints). Therefore,
\begin{align*}
\int_{\Gamma-\beta_0} & \log^2 (|t|+\gamma_0+5) \exp(-Lt^2) |ds|
\\& \le 2 \int_{0}^\infty \log^2 (t+\gamma_0+5) e^{-Lt^2} dt \quad  + \de \sum_{k=1}^{\infty} \log^2 (t_k+\gamma_0+5) e^{-Lt_k^2}.
\end{align*}

Using the standard Vinogradov notation $\ll$ for explicit numerical constants only, for the integral here we easily see
$$
\int_{0}^\infty = \int_{0}^{\gamma_0+5}  +\int_{\gamma_0+5}^\infty \le \log^2(2\gamma_0+10)\int_0^\infty e^{-t^2}dt + \int_5^{\infty} \log^2(2t) e^{-t^2} dt \ll \log^2(\gamma_0+5).
$$
Recalling that by construction $t_1\ge 4$ and $t_k\ge t_{k-1}+1$, ($k\ge 2$), we get a similar estimate for the sum. Therefore,
$$
\int_{\Gamma-\beta_0} \log^2 (|t|+\gamma_0+5) \exp(-Lt^2) |ds| \ll \log^2(\gamma_0+5).
$$
Collecting the above estimates and taking into account $M<\log x$, the definition of $\om$ and that $\ga_0$ is large--needed here in the form that $\log^2(\ga_0+5)< \de^5 \ga_0$--we are led to
\begin{align}\label{Stranslatedintegral}
\bigg| \frac{1}{2\pi i} \int_{\Gamma-\beta_0}  H(s+\rho_0) e^{Ls^2+Ms}ds \bigg| & \le \frac{A_5}{\de^5} \log^2(\gamma_0+5) e^{4L\de^2-\de M + (1-\beta_0)\log x}
\\&\le A_5  e^{4L\de^2-\de M+(1-\beta_0)\log x + \log \ga_0} =A_5 e^{4L\de^2-\de M + \om}. \notag
\end{align}

Next we see to the estimations of the various parts of the right hand side sum of \eqref{Uotherway}.
Keeping the notation $a=\frac{b+\theta'}{2}=1-2\de$ used in the construction of $\Gamma$ let us write
\begin{align}\label{Zfarrootsum}
S_1(\rho_0)&:=\sum^{\star}_{\rho;~ |\Im \rho-\gamma_0|\ge 3} \exp\left(L(\rho-\rho_0)^2+M(\rho-\rho_0)\right)  \notag
\\ \notag & \le \sum_{n=3}^\infty e^{L(\beta_0-a)^2-Ln^2+M(1-\beta_0)} \left\{N(a,\gamma_0-n-1,\gamma_0-n)+ N(a,\gamma_0+n,\gamma_0+n+1)\right\}
\\ & \le e^{L(\beta_0-a)^2+M(1-\beta_0)} \sum_{n=3}^\infty e^{-n^2L} \left(A_6+A_7\log(\gamma_0+n)\right)
\\ & \le A_8 \log(\gamma_0+5) e^{4L\de^2+M(1-\beta_0)-9L} \le A_8  e^{\log \ga_0 +(1-\beta_0)\log x-8L} \le A_8 e^{\om-8L},\notag
\end{align}
referring to Lemma \ref{l:zeroesinrange} in the third line and then calculating similarly as we did above for the sum $\sum_{n=1}^{\infty} \log^2 (t_n+\gamma_0+5) e^{-mt_n^2}$ and in \eqref{Stranslatedintegral}.

For the contribution of the remaining zeroes with $|\Im \rho-\ga_0|\le 3$ in the full sum of residues, let us recall that by construction $(\beta_0,1]\times i[\gamma_0-3,\gamma_0+3]$ is zero-free. In other words, for any such remaining zero it holds $\Re(\rho-\rho_0)\le 0$, always. Using this we can apply a refined estimation for zeroes in a middle range distance from $\rho_0$ as follows.
\begin{align}\label{Zmiddlewayrootsum}
S_2(\rho_0)&:=\sum^{\star}_{\rho;~ 4\de \le |\Im \rho-\gamma_0|\le 3} \exp\left(L(\rho-\rho_0)^2+M(\rho-\rho_0)\right)
\\ \notag & \le \sum_{n=4}^{3/\de} e^{L\left((\beta_0-a)^2-n^2\de^2\right)} \left\{N(a,\gamma_0-(n+1)\de,\gamma_0-n\de)+ N(a,\gamma_0+n\de,\gamma_0+(n+1)\de\right\}
\\ \notag & \le e^{L(\beta_0-a)^2} \sum_{n=4}^{3/\de} e^{-n^2L\de^2} \left(A_6+A_7\log(\gamma_0+5)\right)
 \le A_8 \log\gamma_0 e^{(1-a)^2 L -16\de^2 L}= A_8 e^{\log\log \ga_0-12\de^2 L}.
\end{align}

Combining \eqref{Uotherway}, \eqref{Stranslatedintegral}, \eqref{Zfarrootsum} and \eqref{Zmiddlewayrootsum} and taking into account also \eqref{S-upperesti},
we are led to
\begin{equation}\label{beforeCassels}
|U_2|\ge \left|S_3 \right| - A_9 \left\{e^{\vp\om + 6\om m/\log x-m^2/4L}+e^{\om -0.04 \log^2 x/L}+e^{\log\log \ga_0-12\de^2 L} + e^{4L\de^2+\om-\de M}+ e^{\om-8L}\right\},
\end{equation}
where $S_3:=S_3(\rho_0)$ is defined as
\begin{align*}
S_3& :=\sum^{\star}_{\rho;~ |\Im \rho-\gamma_0|< 4\de} \exp\left(L(\rho-\rho_0)^2+M(\rho-\rho_0)\right).
\end{align*}
Let us denote the number of $\ze$-zeroes in the circle $D(r):=\{ s~:~|s-\rho_0|\le r\}$ as $\nu(r)$. Given that all zeroes in $S_3$ are to the right of $\Gamma$, therefore have $\Re \rho \ge a=1-2\de$, it is clear that $D(5\de)$ covers all the zeroes in the sum for $S_3$. If a parameter $r_0\ge \frac{15\log\log\log \gamma_0}{\log \log \gamma_0}$ is given, then according to the Turán-type Lemma \ref{th:locdens} we have\footnote{For the applicability of the Lemma recall $\de <0.01 (1-\theta)$ i.e. $5\de < 0.05 (1-\theta) \le 0.1(\beta_0-\theta)$ given that $\beta_0>(1+\theta)/2$.} the estimate $\nu(r) \le A_{10} r \log(\gamma_0+5)$ for any $r_0<r<5\de$.

Here we cut the sum $S_3$ into two parts, and estimate the part with zeroes not closer to $\rho_0$ than $r_0$ as follows.
\begin{align*}
R& :=\sum^{\star}_{\rho;~ r_0 \le |\rho-\gamma_0|< 5 \de} \left|\exp\left(L(\rho-\rho_0)^2+M(\rho-\rho_0)\right)\right|
\\ &\le \int_{r_0}^{5\de} \max_{\beta+it \in \partial D(r)} \exp\left(L((\be-\be_0)^2-(t-\gamma_0)^2)-M(\be_0-\be)\right) d\nu(r).
\end{align*}
Here the inner function is $\Phi(r):=\exp\left( \max_{0\le u \le r; ~u^2+v^2=r^2} \{L(u^2-v^2)-Mu \}\right)$ so that $\varphi(r):=\log \Phi(r)= \max_{0\le u \le r} \left( L(2u^2-r^2)-Mu \right)=-r^2L + \max_{0\le u \le r} \left( 2Lu^2-Mu \right)$. This last maximum is a maximum of a convex function, whence is attained at some of the endpoints: at 0 it attains  zero, at $r$ it is $2Lr^2-Mr$, and we get $\Phi(r)=\exp\left(-Lr^2+\max(0,2Lr^2-Mr)\right)$. Obviously, if we only assume that $L\le M$, then for $r\le 5\de <1/2$ we will always have $2Lr^2-Mr < 0$, whence the maximum is $0$ here. This yields
\begin{align}\label{eq:Rupper}\notag
R &\le \int_{r_0}^{5\de} \exp(-Lr^2) d\nu(r) = \left[\exp(-Lr^2) (\nu(r)-\nu(r_0)) \right]_{r_0}^{5\de} + \int_{r_0}^{5\de} 2Lr\exp(-Lr^2) (\nu(r)-\nu(r_0)) dr
\\& \ll \exp(-25L\de^2) \de \log(\gamma_0+5) + \int_{r_0}^{5\de} 2Lr \exp(-Lr^2) \log(\gamma_0+5)r dr \notag
\\& \ll \log\gamma_0 \left\{ \de e^{-25L\de^2} + \int_{r_0}^{\infty} Lr^2 \exp(-Lr^2) dr\right\}
=\log\gamma_0  \left\{ \de e^{-25L\de^2} + \int_{Lr_0^2}^{\infty} v e^{-v} \frac{dv}{2\sqrt{vL}}\right\}
\\& \notag
\le\log\gamma_0  \left\{ e^{-Lr_0^2} + \frac{1}{2Lr_0} \int_{Lr_0^2}^{\infty} v e^{-v} dv \right\} =e^{\log\log\gamma_0}  \left\{ 1+\frac{1+Lr_0^2}{2Lr_0} \right\} e^{-Lr_0^2}\ll e^{\log\log \ga_0 -Lr_0^2}.
\end{align}

There remains the part $P$ of $S_3$ with the remaining zeroes $|\rho-\rho_0|\le r_0$ (and to the right of $\Gamma$). Their number is $n \ll r_0\log(\gamma_0+5)$. Now, $P$ can be written as a sum of pure powers (i.e. without coefficients), where the general term takes the form
$$
\exp\left(L(\rho-\rho_0)^2+M(\rho-\rho_0)\right)=\exp\left(M \left(\ell (\rho-\rho_0)^2+(\rho-\rho_0)\right)\right)=e^{M\lambda(\rho)},
$$
with $\lambda(\rho):=\ell (\rho-\rho_0)^2+(\rho-\rho_0)$ and $\ell:=L/M$. In the following we will fix $\ell$ as a constant, i.e. take $L$ and $M$ to be constant multiples of each other, with $M$ varying in $[\log x -2m, \log x -m]$, but $\ell$ fixed. Out of terms of $P$ there is one, belonging to $\rho_0$, which must be exactly 1. Therefore, Turán's Second Main Theorem of the Power Sum Theory, Lemma \ref{l:SecondMain} above, gives that in the interval given for $M$ there is a choice of the value of this parameter which furnishes
\begin{equation}\label{eq.Ppowersum}
|P|\ge \exp\left( -\log\left(\frac{8e\log x}{m}\right)\cdot A_{11} r_0 \log(\gamma_0+5) \right)
\end{equation}
and consequently in view of \eqref{eq:Rupper}
\begin{equation}\label{eq:S3below}
|S_3|\ge \exp\left( -\log\left(\frac{8e\log x}{m}\right)\cdot A_{11} r_0 \log(\gamma_0+5) \right) - A_{12} \exp(\log\log \ga_0-r_0^2 L).
\end{equation}

\section{Choice of parameters and proof of Theorem \ref{th:Beurlingdomainosci}}\label{sec:parametersandproof}

As a key step towards the oscillation result on $\D(x)$, we first deduce the following intermediate result.

\begin{theorem}\label{th:goodlocalization} Let $\ve >0$ be arbitrary. Assume Axiom A and $\theta_0=1$. If $x>x_0(\G,\ve)$ and $\rho_1$ is the extremal $\zeta$-zero for $x$ (i.e. $\om(\rho_1; x)=\om(x)$), then there exists some $\xi \in [x^{1-\ve},x]$ such that $|\D(\xi)|\ge \xi^{\be_1}/|\rho_1|^{1+\ve}$.
\end{theorem}

\begin{remark} One would think that the result should hold for \emph{any} zero, once it holds for the extremal one. However, $\rho_1$ is extremal \emph{only for} $x$, while it is not clear what would be the right comparison with the extremal zero \emph{right for} $\xi$. Below in our main result we establish such a comparison--at the expense of a loss of $e^{\ve\om(\xi)}$, slightly exceeding the mere $|\rho|^\ve$ loss here. That is quite satisfactory, yet we should admit that the plausible variant here for arbitrary zeroes we could not prove.
\end{remark}

\begin{proof} As $\theta_0=1$, $\Re \rho=\theta_0$ is not attained by any zero. Therefore, as $x\to \infty$, the extremal zeroes in $\ot \approx \om$ satisfy "$\rho \to 1+ i\infty$", as is discussed in Lemma \ref{l:omega}. Thus we can assume $\ga_1$ to be sufficiently large whenever $\ga_1=\Im \rho_1$ for an $x$-extremal $\rho_1$ with large enough $x$.

Recall from the above argument that we replaced the extremal zero $\rho_1$ by some possibly different $\rho_0$ with $\gamma_1\le \gamma_0\le (1+o(1))\gamma_1$ and $\beta_0\ge \beta_1$, and such that $\om(\rho_0; x)\le \om(x)+1$. Obviously, it suffices then to prove the assertion for this possibly changed zero $ß\rho_0$, for $\be_0\ge \be_1$ and $|\rho_0|=|\rho_1|(1+o(1))$ entails
$\frac{\xi^{\be_0}}{|\rho_0|^{1+\ve}}\ge  \frac{\xi^{\be_1}}{[(1+o(1))|\rho_1|]^{1+\ve}} \ge \frac{\xi^{\be_1}}{|\rho_1|^{1+2\ve}}$, say.

We apply the results of the calculations of the above Sections \ref{sec:upperesti} and \ref{sec:loweresti} with $\de$ chosen in the beginning of Section \ref{sec:upper} and with the further specifications of our parameters
\begin{equation}\label{eq:parameters}
\vp<\de^2, \quad \ell:=\frac{1}{8} ~(\textrm{i.e.}~ L:=\frac{1}{8} M),\quad m:=\vp \log x, \quad r_0:=3\vp^2.
\end{equation}
Given that $\gamma_0\ge\gamma_1$, we certainly know that $\ga_0\to \infty$, thus for large enough $x$ the above choice of $r_0=3\vp^2$ exceeds $15\log\log\log \ga_0/\log \log \ga_0$, whence is admissible. Similarly, as $\om\in [\om(x),\om(x)+1]$ and $\om(x)/\log x\to (1-\theta_0)=0$, we certainly have $\ga_0\le \om =o(\log x)$. So, for large enough $x$ we also have $\log \ga_0 \le \om \le \frac{1}{3} \vp^3 \log x$. Further, from $\ga_0\to \infty$ it also follows that for large enough $\gamma_0$--that is, for large enough $x$--we necessarily have $\log \log \gamma_0 <\vp \log \ga_0$, and in particular even $r_0^2L \ge \vp^4 \log x \ge \vp \log \ga_0>\log\log\ga_0 >1$ and also $\log \gamma_0> \frac{1}{\vp} \log\log\gamma_0>1/\de^2 >10 \log(1/\de)$

With the above choice of parameters the conditions that  {$m\le |\rho_0| L$}, that {$m>4\sqrt{L}$, $m>24L\om/\log x$, $M\ge 0.95 \log x$} and that  {$\log x \ge 20 \sqrt{L} $} are obviously met, thus from
\eqref{beforeCassels} we are led to
\begin{align}\label{eq:Uwithparameters}
\notag |U_2| & \ge |S_3|-A_9 \left\{e^{7\vp\om -\vp^2 \log^2 x/4L}+e^{\om -0.04 \log^2 x/(\frac18 \log x)}+e^{\vp \log \ga_0-12\vp L} + e^{\om+4\de^2 L -\de M}+ e^{-7L}\right\}
\\& \notag \ge |S_3|- A_9 \left\{e^{7\vp\om -2\vp^2 \log x}+e^{\om -0.3 \log x}+e^{-11\vp L} + e^{\om  - \de M/2}+ e^{-7L}\right\}
\\& \ge |S_3|- O\left(e^{-2\vp\om}+e^{-\om}+e^{-11\vp \om} + e^{\om-\vp \log x}+ e^{-7\om}\right)
=|S_3|-O(e^{-2\vp \om})=|S_3|- O(\ga_0^{-2\vp}).
\end{align}
Further, substituting our parameter choices into \eqref{eq:S3below} yields for $\vp$ sufficiently small in function of $A_{11}$
\begin{align}\label{eq:S3withparameters}
|S_3| &\ge \exp\left( -\log\left(\frac{8e}{\vp}\right)\cdot A_{11} 3\vp^2 2 \log \gamma_0 \right) - A_{12} \exp(\log\log \ga_0-9\vp^4 L)
\\ &\ge e^{-\vp\log \gamma_0} - A_{12}e^{\vp \log\ga_0-\vp^4 \log x }\ge e^{-\vp\log \gamma_0} - A_{12}e^{\vp \log\ga_0-3\vp \log \ga_0 }\ge \ga_0^{-\vp} -O(\ga_0^{-2\vp}).\notag
\end{align}
Collecting \eqref{eq:U2upperDelta}, \eqref{eq:Uwithparameters} and \eqref{eq:S3withparameters} finally leads to
$$
\max_{\xi \in [q,x]} \frac{|\D(\xi)|}{\xi^{\beta_0}/|\rho_0|} \gg \ga_0^{-\vp}\left(1-O(\ga_0^{-\vp})\right),
$$
proving the assertion if $\vp$ is chosen sufficiently smaller than $\ve$.
\end{proof}

Now we are in the position to infer the second main result of the present study.

\begin{theorem}\label{th:Deltaomega} Assume Axiom A and $\theta_0=1$.

If $0<\ve<0.01(1-\theta)$ is arbitrary and if $x>x_0(\G,\ve)$, then there exists some $\widetilde{x} \in [x^{1-\ve},x]$ such that $|\D(\widetilde{x})|\ge \widetilde{x} \exp(-(1+\ve)\om(\widetilde{x}))$.
\end{theorem}

\begin{proof}
We apply the above Theorem \ref{th:goodlocalization}: there is an $\widetilde{x} \in [x^{1-\ve},x]$ such that
\begin{equation}\label{eq:LastThFirstStep}
|\D(\widetilde{x})| \ge \frac{\widetilde{x}^{\be_1}}{|\rho_1|^{1+\ve}} = \frac{\widetilde{x}}{\widetilde{x}^{1-\be_1} |\rho_1|^{1+\ve}} \ge \frac{\widetilde{x}}{(\widetilde{x}^{1-\be_1} |\rho_1|)^{1+\ve}}
\ge \frac{\widetilde{x}}{(x^{1-\be_1} |\rho_1|)^{1+\ve}}=\frac{\widetilde{x}}{\exp((1+\ve)\om(x))}.
\end{equation}
Next we want to compare $\om(x)$ and $\om(\widetilde{x})$ in the exponent. That is, we write
\begin{equation}\label{eq:LastThSecondStep}
|\D(\widetilde{x})| \ge \frac{\widetilde{x} \exp(-(1+\ve)\om(\widetilde{x}))}{\exp((1+\ve)(\om(x)-\om(\widetilde{x})))}.
\end{equation}
We can write in the denominator according to \eqref{omegayz} of Lemma \ref{l:omega} that $\om(x)-\om(\widetilde{x})\le \om(\widetilde{x}) \frac{\log x - \log \widetilde{x}}{\log \widetilde{x}} \le \frac{\ve}{1-\ve} \om(\widetilde{x})$, and that gives $|\D(\widetilde{x})| \ge \widetilde{x} \exp(-\frac{1+\ve}{1-\ve}\om(\widetilde{x}))$, which is obviously sufficient as $\ve$ was arbitrary.
\end{proof}

Contrary to the upper estimations in the direct direction of Ingham type theorems, here in the converse direction the analogous results with the Ingham function $\eta(t)$ are not so immediate. The reason is that we have seen that if $\eta$ is to bound a domain $D(\eta)$ free of zeroes, then $\om(x)\ge\om_\eta(x)$. However, when $\eta$ is a curve with $D(\eta)$ \emph{containing an infinitude of zeroes}, then in general \emph{no global comparison} of the two functions $\om$ and $\om_\eta$ is available. Indeed, it is possible to give examples of functions with $\om(x)>\om_\eta(x)$ and also $\om(x)<\om_\eta(x)$ occurring at many places. Therefore, what we will do in the below proof is a pointwise comparison, to our favor, i.e. in the direction of $\om_\eta(x)\ge \om(x)$, \emph{only along a sequence}, defined geometrically and indirectly by the outstanding sequence of zeroes $\rho_k$. The construction is non-trivial, but is necessary, if we indeed want to show that the relatively recent versions of Pintz' Theorem using the function $\om(x)$ directly defined in terms of the zeroes rather than Ingham's $\om_\eta(x)$, are indeed sharper. As said, this is clear for the direct direction but is surprisingly less obvious, requiring also some convex geometry for the converse direction.

\begin{proof}[Proof of Theorem \ref{th:Beurlingdomainosci}] So let us assume now that $\rho_k$ is a sequence of zeroes in the domain \eqref{eq:etazerofree}. We assume that $\eta(t)$ is convex in logarithmic variables, and also that $\theta_0=1$, all other cases being similar. So we assume $\eta(e^v)$ being convex and consider the domain $\Ds:=\{(u,v)\in\RR^2~:~ u\ge 1-\eta(e^v)\}$, which is the image of $D(\eta)$ under the canonical mapping\footnote{In the definition of $\eta$ we restricted to $ßt\ge 1$. That corresponds to set the domain of $\Phi$ as $\Im s\ge 1$, fitting to taking $\log t$ and considering the upper halfplane here; also this fits to the restriction applied in the definition of $\ot(x)$.} $\Phi: \CC \to \RR^2$ defined by $\Phi(\si+it)\to (\si,\log t)$. In fact, all our considerations will take place in the upper halfplane $\HH:=\RR\times[0,\infty)$, as $\Phi$ maps the halfplane $S:=\{s=\si+it~:~ t\ge 1\}$ to $\HH$. We define the boundary curve $\C:=\Cn\cup \Co$, where $\Cn:=(-\infty,1-\eta(1)]\times\{0\}$, a halfline on the $u$-axis of $\RR^2$, and $\Co:=\{(u,v)\in \HH~:~ u=1-\eta(e^v)\}$. The latter is the boundary curve for $\Ds$ at least under the convention that we decide to consider $D(\eta)$ as only in its part belonging to $S$. Similarly, we assume that $\eta$ satisfies $\eta(1)<1-\theta$, (a condition automatically satisfied in the classical case when assuming $\eta\le 1/2$).

It will be more convenient to describe $\Co$ by means of the inverse function\footnote{It exists, as $\eta$ tending to $0$ while $\eta(e^v)$ being convex entail that $\eta$ has to be strictly decreasing.} $\ff:=f^{-1}$, where $f(v):=1-\eta(e^v)$. Then $f: [0,\infty)\to [a,1)\subset (\theta,1)$ with $a:=1-\eta(1)>\theta$, whence $\ff:[a,1)\to [0,\infty)$. Note also that by assumption $\eta(e^v)$ is convex, whence $f$ is concave, and $\ff$ is again convex. Now with $\ff$ we can write $\Co:=\{ (u,v)~:~ v=\ff(u)\}$.

Consider the sequence $\rho_k^*:=\Phi(\rho_k):=(u_k,v_k):=(\beta_k, \log\gamma_k) \in \HH$. We
define the convex hull
$$
K:=\con\left((\HH\setminus \Ds)\cup (\cup_{k=1}^\infty \{\rho_k^*\})\right)=\con \left( \C \cup (\cup_{k=1}^\infty \{\rho_k^*\})\right).
$$
Here in making the convex hull, points which are not extreme points can be dropped without changing the set $K$. It is important that deleting these non-extreme points, there still remains an infinite sequence of zeroes $\rks$. We will now prove that if there were only a finite number of extreme points among the $\rho_k^*$ , then there had been only a finite number of points of the sequence, too.

Indeed, if there are only finitely many extreme points among the $\rks$, then they belong to a certain rectangle $R:=[a,U]\times [0,V]$, say, with $U<1$, because $\Re s=1$ is zero-free. Then it is easy to see that all the supporting lines to $\Co$ emanating from points of $R$ are at most as steep as the one emanating from $(U,0)$. Let this last supporting line touch $\Co$ at $\ww:=(w,\ff(w))$ and have slope $\mu$, say. In other words the line $v=\mu(u-U)$ is a supporting line to $\Co$, supporting $\Co$ at $\ww$.

Now it is easy to see that the whole rectangle $R$ is to the left of this line (because it emanates from the point $(U,0)$ with the rightmost coordinate $U$ of $R$, and its slope is positive, given that it must pass from the right all points of $\Co$, those with abscissa arbitrarily close to 1 included), therefore this line is supporting on its left both $R$ and $\Co$, whence also the whole of $K$. Normally (at points of differentiability) $\mu$ is just the derivative $\ff'(w)$, but this is not needed. What is needed is that at any point $\zz:=(z,\ff(z))$ with $z>w$ all supporting lines of $\Co$ through $\zz$ have at least as large a slope $\mu(z)$ than $\mu$. Any such supporting line $v=\mu(z)(u-z)+\ff(z)$ has $\ww$ on its left (because it supports $\Co$), and is steeper (not less steep) than $\mu$, whence it also has all the rectangle $R$ on its left. Therefore, it is a supporting line to $K$, too.

Let now $(p,q)\in\HH$ be any point with $q>\ff(p)$ and $p>w$. According to the above, the supporting line to $\Co$ at $(p,\ff(p))$ is a supporting line to $K$, too. However, $(p,q)$ is to the right of this line (as $q>\ff(p)$), therefore it is separated from $K$ by this line and therefore it cannot belong to $K$. In particular, it cannot be any element of the sequence $\rks$ (as all of these belong to $K$). That is, we do not have any point $\rks$ in the domain $\Ds$ with $p>w$. It follows that only points $(p,q)$ with $a\le p \le w$ can be equal to some $\rks$, and therefore the sequence $\rks$ itself had to be finite\footnote{The set is part of $\Z(a;1,e^w)$, whose cardinality is $N(a;1,e^w)$.}.

So we now see that there has to be an infinite (sub)sequence of the $\rks$ which are also extreme points\footnote{Note on passing that here we essentially used convexity of $\eta(e^v)$, equivalent to convexity of $\ff$, for otherwise it is easy to construct examples with an infinite sequence $\rks \in \Ds$ but none of them being extreme points of $K$.} for $K$: we delete others and keep only these extreme points.


Now if $\rho_k^*$ is an extreme point of $K$ then there exists a supporting line through it for the convex set $K$, i.e. a line $L_k$ with defining equation $v=a_ku+b_k$, which
passes through the point $\rks:=u_k+iv_k$, while it goes fully outside the interior of the convex hull $K$.

Take the straight line $\ell$ passing through the point $(1,0)$ and of slope $\xi$, (a positive parameter to be chosen later on) in the $(u,v)$-plane $\RR^2$. Its equation is $v=\xi(u-1)$, or $\xi u-v-\xi=0$. In the $(u,v)$-plane, the line $\ell$ is to the right\footnote{Indeed, its part in $\HH$ lies in the quadrant $u\ge 1, v\ge 0$, while $K$ lies in the open quadrant $u<1, v\ge 0$.} of the above defined convex domain $K$. Substituting the coordinates of any point in the normal form $\frac{1}{\sqrt{1+\xi^2}} (\xi u-v-\xi)=0$ of $\ell$ furnishes the signed distance of that point from $\ell$, with points to the left from $\ell$ bearing negative, and points to the right having positive signs. This signing being some inconvenience for us, we change the orientation: we write instead $\ell(u,v):=\frac{1}{\sqrt{1+\xi^2}} (-\xi u+v+\xi)$ for the formula of the line, meaning that now $\ell(u,v)=0$ describes points of $\ell$, and $\ell(u,v)$ for a general point gives the negatively signed distance of the point from $\ell$ i.e. positive signs to the left, and negative signs to the right of $\ell$.

Now minimizing over all points of $K$ corresponds to finding the distance of $K$ from $\ell$. Further, if $(u_0,v_0)\in \partial K$ is one such minimal distance point, then the line $L$ passing through $(u_0,v_0)$ and parallel to $\ell$ is just a supporting line to $K$.

That means in particular that in case we set $\xi:=a_k$, then the prescribed slope matches the slope of the supporting line drawn to $\rho_k^*$, whence the above constructed line $L$ will match $L_k$. As said, it is a supporting line to $K$ through $\rho_k^*$, which also means that all points of $K$, in particular all points on the curve $\Co$, will have to lie to the left from it. The distance to the left from $\ell$ is now signed positively, and $\rks$ provides the least such distance as compared to all points of $K$; in particular it provides a smaller (not larger) value than minimizing over all points of $\Co$, and also it provides the minimal value among all points $\rho_n^*$. So in particular the distance $\ell(u_k,v_k)$ is minimal among all $(u_n,v_n)$: equivalently,
$(1-u_k)a_k+v_k =\min_{n\in \NN} (1-u_n)a_k+v_n$. We therefore find that the slope $a_k$ of the supporting line through the extreme point $\rks$ defines a value $x_k:=e^{a_k}$ such that $\om(\rho_k; x_k)=(1-\beta_k)\log x_k + \log \gamma_k=(1-u_k)a_k + v_k =\min_{n\in \NN} \om(\rho_n; x_k) =\om (x_k)$.

Further, we can read the meaning of the other distance minimization: for any point $(u,v)\in \Co$, we have $\frac{1}{\sqrt{1+\xi^2}} \om(\rho_k; x_k)= \ell(\rks)\le \ell(u,v)=\frac{1}{\sqrt{1+\xi^2}}\left( (1-u)a_k+v\right)$. Multiplying by $\sqrt{1+\xi^2}$ and taking minimum on the right hand side furnishes $\om(x_k)=\om(\rho_k,x_k) \le \min_{(u,v)\in\Co} (1-u)a_k+v = \min_{t\ge 1} (1-\eta(t))\log x_k+\log t =\om_\eta(x_k)$.

So we get that setting $x_k:=e^\xi=e^{a_k}$, it holds $\om_\eta(x_k)\ge \om(\rho_k; x)=\om(x_k)$. This is not a general equation, it is conditional to the property that $\rks$ is an extreme point with the supporting line through it having slope $\xi:=\log x_k=a_k$, but it holds for arbitrary values of $x=x_k$ to which there is a supporting line through a point $\rks$. As said, we already know that there is an infinitude of $\rks$ being extreme points of the convex hull $K$, therefore there is an infinite sequence of $\xi_k=\log x_k$ values for which the lines with slopes $\xi_k$ and passing through $\rks$ provide supporting lines to $K$. With that sequence of $x_k$ we therefore have that $\om(x_k)=\om(\rho_k,x_k) \le \om_\eta(x_k)$. It is clear that these slopes are tending to infinity, therefore also $x_k=e^{\xi_k} \to \infty$.

Here we can apply Theorem \ref{th:Deltaomega} (with $x:=x_k$), more precisely its proof, where in the first step we have got \eqref{eq:LastThFirstStep}. We obtain values
$\widetilde{x_k} \in [x_k^{1-\ve},x_k]$, for which $|\D(\widetilde{x_k})|\ge \widetilde{x_k} \exp(-(1+\ve)\om(x_k))\ge \widetilde{x_k} \exp(-(1+\ve)\om_\eta(x_k))$, the last step being clarified above.

Note that $\om_\eta$ satisfies similar properties as $\om$--see Lemma \ref{l:omegaeta}. That means that we can argue as above in the end of the proof of Theorem \ref{th:Deltaomega}, finally deriving  $|\D(\widetilde{x_k})|\ge \exp(-\frac{1+\ve}{1-\ve}\om_\eta(\widetilde{x_k}))$.
\end{proof}

\section{Concluding remarks}\label{sec:Conclusion}

In the paper and in the previous parts of the series we obtained several results on the order of magnitude of the error function $\D(x)$ in the Beurling PNT under Axiom A
In particular, we emphasized the role of the Riemann-von Mangoldt formula, which seems to suggest reliable information on the order of magnitude of the error function even if it is a divergent, hardly manageable series, whose "interference difficulty" was truthfully pointed out by Littlewood. Even if this series is divergent--so that in concrete analysis only truncated versions, i.e. partial sums can be used--and even if it is not a regular say Taylor series, its characteristics seem to have many resemblance to well-behaved entire functions of say finite order, such as its order of growth being comparable to the size of its largest term, or to the total sum of the absolute values of terms. Therefore, in our opinion, reasonable conjectures can be made by analogy to power series, and the challenge lies in extracting the conjectured information from the otherwise complicated series and indeed prove what seems to be suggested by it.


We need to mention another direction here, which seems to originate from a paper by Knapowsky \cite{Knapowski} in the classical case. Namely, from assuming a known $\zeta$-zero, he obtained a lower estimate for the mean value $D(x)$, too. Also this was improved to very sharp, essentially optimal forms by Pintz \cite{PintzBasel} \cite{PintzNoordwijkerhout} \cite{PintzProcStekl}, made suitable also for handling dependence on zero-free regions (and not only on one pre-set zero) similar to our topic here. That enabled Pintz to give an essentially full picture of the essentially equivalent order of magnitude of the functions $\D(x)$, its average $ßD(x)$, its maximum $S(x)$, and the zero-distribution dependent functions $W(x)$ and $Z(x)$, see \cite{PintzProcStekl}.

However, we cannot fully extend Pintz' result on the lower estimation of the average $D(x)$ to the Beurling case. The main reason is that Pintz used Vinogradov estimates heavily, while we found surpassing that part of the argument very hard, essentially not working.  Nevertheless, already here we employed certain finer calculus and estimates at some steps, than were merely necessary for the proofs of our statements right here, in the good hope that these finer estimates will be of good use in further investigations--in particular concerning $D(x)$.

\section{Acknowledgement}

We would like to thank János Pintz, Frederik Broucke, Gregory Debruyne and the anonymous referee for many useful comments and advices, which were benefitted throughout our work and its subsequent revision. Also we would like to thank Ghent University for the hospitality and the motivating professional environment, which was greatly stimulating during our visit in May 2022. Also the financial support of the BOF grant, project no. 01J04017 is thankfully acknowledged.
During the fifteen years of writing this paper the author was supported by the Hungarian Research, Development and Innovation Office, Grant \# T-72731, T-049301, K-61908, K-119528 and K-132097.

\renewcommand\refname{

\end{document}
\centerline{Literature}}


\begin{thebibliography}{AAAA}

\bibitem{H-15}{\sc Al-Maamori, F.; Hilberdink, T. W.}
An example in Beurling's theory of generalised primes.

\emph{Acta Arith.} {\bf 168} (2015), no. 4, 383--396.


\bibitem{Aramaki} {\sc Aramaki, Junichi}, On an extension of the Ikehara Tauberian theorem, \emph{Pacific J. Math.} {\bf 133} (1988), 13--30.

\bibitem{Beke} {\sc T. Bekehermes}, {\em Allgemeine Dirichletreihen und Primzahlverteilung in arithmetische Halbgruppen}, Dissertation zur Erlagung des Grades eines Doktors, Technischen Universit\"at Clausthal, 2003.

\bibitem{Beur}
{\sc A. Beurling}, Analyse de la loi asymptotique de la
distribution des nombres premiers g\'en\'eralis\'es I. \emph{Acta
Math.} {\bf 68} (1937), 255--291.



\bibitem{Broucke}{\sc F. Broucke}, Note on a conjecture of Bateman and Diamond concerning the abstract PNT with Malliavin-type remainder, \emph{Monatsh. Math.} {\bf 196} (2021), no. 3, 457--470.

\bibitem{BrouckeDebruyne}{\sc F. Broucke and G. Debruyne}, On zero-density estimates and the PNT in short intervals for Beurling generalized numbers. To appear in \emph{Acta Arith.}. See also at {\tt https://arxiv.org/abs/2211.08716}.

\bibitem{BrouckeVindas}{\sc F. Broucke, J. Vindas}, A new generalized prime random approximation procedure and some of its applications, ArXiv preprint no. arxiv:arXiv:2102.08478, see at {\tt https://arxiv.org/abs/arXiv:2102.08478v1}.


\bibitem{BrouckeDebruyneVindas}{\sc F. Broucke, G. Debruyne  } and {\sc J. Vindas}, Beurling integers with RH and large oscillation. \emph{Adv. Math.} {\bf 370} (2020), Article no. 107240, 38 pp.


\bibitem{DebruyneVindas-PNT}{\sc G. Debruyne } and {\sc J. Vindas},
On PNT equivalences for Beurling numbers.
\emph{Monatsh. Math.} {\bf 184} (2017), no. 3, 401--424.

\bibitem{DebruyneVindas-RT}{\sc G. Debruyne } and {\sc J. Vindas}, On general prime number theorems with remainder. (English summary) \emph{Generalized functions and Fourier analysis}, 79--94, in:
Oper. Theory Adv. Appl., {\bf 260}, Adv. Partial Differ. Equ. (Basel), Birkhäuser/Springer, Cham, 2017.

\bibitem{DebruyneVindas}{\sc G. Debruyne } and {\sc J. Vindas}, On Diamond's $L^1$ criterion for asymptotic density of Beurling generalized integers. \emph{Michigan Math. J.} {\bf 68} (2019), no. 1, 211--223.

\bibitem{DebruyneDiamondVindas}{\sc G. Debruyne, H. Diamond } and {\sc J. Vindas}, $M(x)=o(x)$ estimates for Beurling numbers. \emph{J. Th\'eor. Nombres Bordeaux} {\bf 30} (2018), no. 2, 469--483.

\bibitem{DSV} {\sc Debruyne, G.; Schlage-Puchta, J.-C.; Vindas, J.},
Some examples in the theory of Beurling's generalized prime numbers.
\emph{Acta Arith.} {\bf 176} (2016), no. 2, 101--129.

\bibitem{DMV}{\sc H. G. Diamond, H. L. Montgomery  } and {\sc U. Vorhauer},
Beurling primes with large oscillation, \emph{Math. Ann.}, {\bf 334} (2006) no. 1,  1--36.

\bibitem{Diamond-77}{\sc H. G. Diamond}, When do Beurling generalized integers have a density?, \emph{J. Reine Angew. Math.} {\bf 295} (1977), p. 22--39.

\bibitem{DZ-13-3} {\sc Diamond, H. G.; Zhang, Wen-Bin},
Optimality of Chebyshev bounds for Beurling generalized numbers.
\emph{Acta Arith.} {\bf 160} (2013), no. 3, 259--275.

\bibitem{DZ-13-2} {\sc Diamond, H. G.; Zhang, Wen-Bin},
Chebyshev bounds for Beurling numbers.
\emph{Acta Arith.} {\bf 160} (2013), no. 2, 143--157.

\bibitem{DZ-17} {\sc Diamond, H. G.; Zhang, Wen-Bin},
Prime number theorem equivalences and non-equivalences.
\emph{Mathematika} {\bf 63} (2017), no. 3, 852--862.

\bibitem{DZ-16} {\sc Diamond, H. G.; Zhang, Wen-Bin}, \emph{Beurling generalized numbers},
Mathematical Surveys and Monographs, 213. American Mathematical Society, Providence, RI, 2016. xi+244 pp.


\bibitem{H-19}{\sc Hilberdink, T. W.}
Power moments and value distribution of functions. (English summary)
Trans. Amer. Math. Soc. 371 (2019), no. 1, 1–31.

\bibitem{H-12}{\sc Hilberdink, T. W.}
Generalised prime systems with periodic integer counting function.
\emph{Acta Arith.} {\bf 152} (2012), no. 3, 217--241.

\bibitem{H-10}{\sc Hilberdink, T. W.}
$\Omega$-results for Beurling's zeta function and lower bounds for the generalised Dirichlet divisor problem.
\emph{J. Number Theory} {\bf 130} (2010), no. 3, 707--715.

\bibitem{H-7}{\sc Hilberdink, T. W.; Lapidus, M. L.}
Beurling zeta functions, generalised primes, and fractal membranes.
\emph{Acta Appl. Math.} {\bf 94} (2006), no. 1, 21--48.

\bibitem{H-5}{\sc Hilberdink, T. W.}
Well-behaved Beurling primes and integers.
\emph{J. Number Theory} {\bf 112} (2005), no. 2, 332--344.


\bibitem{Ingham} {\sc A. E. Ingham},
{\it The distribution of prime numbers}, Cambridge University
Press, 1932.

\bibitem{Ingham-AA} {\sc A. E. Ingham}, A note on the distribution of primes. \emph{Acta Arith.} {\bf 1} (1936), 201--211.

\bibitem{Johnston} {\sc D. Jonhnston}, Improving bounds on prime counting functions by partial verification of the Riemann hypothesis,  \emph{Ramanujan J.} {\bf 59} (2022), no. 4, 1307--1321.




\bibitem{K-98} {\sc Kahane, J.-P.}, Le rôle des algebres $A$ de Wiener, $A^\infty$ de Beurling et $H^1$ de Sobolev dans la théorie des nombres premiers généralisés, \emph{Ann. Inst. Fourier} (Grenoble) {\bf 48}(3) (1998) 611--648.

\bibitem{K-99} {\sc Kahane, J. P.},
Un théoreme de Littlewood pour les nombres premiers de Beurling. (French) [A Littlewood theorem for Beurling primes]
\emph{Bull. London Math. Soc.} {\bf 31} (1999), no. 4, 424--430.

\bibitem{K-17} {\sc Kahane, J.-P.},
Conditions pour que les entiers de Beurling aient une densité. \emph{J. Théor. Nombres Bordeaux} {\bf 29} (2017), no. 2, 681--692.


\bibitem{Knapowski} {\sc S. Knapowski} On the mean values of certain functions in prime number theory, {\it Acta Math. Acad. Sci. Hungar.} {\bf 10} (1959), 375--390.

\bibitem{Knopf} {\sc J. Knopfmacher},  {\em Abstract analytic number theory}, North Holland \& Elsevier, Amsterdam--Oxford \& New York, 1975.  (Second edition: Dover Books on Advanced Mathematics. Dover Publications, Inc., New York, 1990. xii+336 pp.)


\bibitem{Littlewood-JLMS} {\sc J. E. Littlewood}, Mathematical notes
(12). An inequality for a sum of cosines, {\it J. London Math.
Soc.} {\bf 12} (1937), 217--222.

\bibitem{Littlewood-CR} {\sc J. E. Littlewood}, Sur la distribution des nombres premiers, C. R. Acad. Sci.
Paris {\bf 158} (1914) 1869--1872.

\bibitem{Mont} {\sc H. L. Montgomery},  {\em Topics in multiplicative
number theory}, Lecture Notes in Mathemtics, {\bf 227}, Springer,
1971.

\bibitem{MT} {\sc M. J. Mossinghof, T. S. Trudgian}, Improving bounds on prime counting functions by partial verification of the Riemann hypothesis,  \emph{J. Number Theory} {\bf 157} (2015), 329--349.



\bibitem{H-20}{\sc Neamah, A. A.; Hilberdink, T. W.}
The average order of the Möbius function for Beurling primes.
\emph{Int. J. Number Theory} {\bf 16} (2020), no. 5, 1005--1011.

\bibitem{Pintz1} {\sc J. Pintz}, On the remainder term of the prime number formula. I. On a problem of Littlewood, {\it Acta Arith.} {\bf 36} (1980), 341--365.

\bibitem{Pintz2} {\sc J. Pintz}, On the remainder term of the
prime number formula. II. On a theorem of Ingham, {\it Acta
Arith.} {\bf 37} (1980), 209--220.

\bibitem{Pintz9} {\sc J. Pintz}, Elementary methods in the theory
of $L$-functions IX. Density theorems {\em Acta Arith.} {\bf XLIX}
(1980), 387-394.

\bibitem{PintzProcStekl} {\sc J. Pintz} Distribution of the zeroes of the Riemann zeta function and oscillations of the error term in the asymptotic law of the distribution of prime numbers. (Russian) \emph{Tr. Mat. Inst. Steklova} {\bf 296} (2017), Analiticheskaya i Kombinatornaya Teoriya Chisel, 207--219. English version published in \emph{Proc. Steklov Inst. Math.} {\bf 296} (2017), no. 1, 198--210.


\bibitem{PintzNewDens} {\sc J. Pintz}, Some new density theorems for Dirichlet L-functions. \emph{Number theory week 2017, Banach Center Publ.}, {\bf 118}, 231--244. Polish Acad. Sci. Inst. Math., Warsaw, 2019. See also at {\tt https://arxiv.org/abs/1804.05552}.


\bibitem{PintzBasel} {\sc Pintz, J.}, Oscillatory properties of the remainder term of the prime number formula. \emph{in: Studies in pure mathematics}, 551--560, Birkhäuser, Basel, 1983. 

\bibitem{PintzNoordwijkerhout} {\sc Pintz, J.},
On the remainder term of the prime number formula and the zeroes of Riemann's zeta-function. \emph{in: Number theory, Noordwijkerhout 1983 (Noordwijkerhout, 1983)}, 186--197,
\emph{Lecture Notes in Mathematics}, {\bf 1068}, Springer, Berlin, 1984. 

\bibitem{PintzMP} {\sc Pintz, J.}, On the mean value of arithmetic error terms, \emph{Math. Pann.}, {\bf 28 (NS 2)} (2022) no. 1, 58--64.


\bibitem{Trud2} {\sc D. J. Platt and T. S. Trudgian}. The error term in the prime number theorem. \emph{Math. Comp.} {\bf 90} (2021), no. 328, 871--881.

\bibitem{Trud3} Platt, D., Trudgian, T.: The Riemann hypothesis is true up to $3\cdot 10^{12}$. \emph{Bull. Lond. Math. Soc.} {\bf  53} (3) (2021), 792--797.


\bibitem{Rev1} {\sc Sz. Gy. R\'ev\'esz}, Irregularities in the
distribution of prime ideals. I, {\it Studia Sci. Math. Hungar.}
{\bf 18} (1983), 57--67.

\bibitem{Rev2} {\sc Sz. Gy. R\'ev\'esz}, Irregularities in the
distribution of prime ideals. II, {\it Studia Sci. Math. Hungar.}
{\bf 18} (1983), 343--369.

\bibitem{RevPh} {\sc Sz. Gy. R\'ev\'esz}, On a theorem
of Phragm\`en,  {\em Complex analysis and applications '85,
(Proceedings of the conference held in Varna, Bulgaria, 1985)},
Publ. House Bulgar. Acad. Sci. Sofia, 1986, p. 556--568.

\bibitem{RevAA} {\sc Sz. Gy. R\'ev\'esz}, Effective oscillation
theorems for a general class of real-valued remainder terms, {\em
Acta Arith.} {\bf XLIX} (1988), 482-505.


\bibitem{RevB} {\sc Sz. Gy. R\'ev\'esz}, On Beurling's prime number theorem,  {\em Periodica Math. Hungar.} {\bf 28} (1994), no. 3, 195--210.

\bibitem{Rev-L} {\sc Sz. Gy. R\'ev\'esz}, On some extremal problems of Landau. \emph{Serdica Math. J.} {\bf 33} (2007), no. 1, 125--162.

\bibitem{Rev-MP} {\sc Sz. Gy. R\'ev\'esz}, A Riemann-von Mangoldt-type formula for the distribution of Beurling primes. \emph{Math. Pann.} {\bf 27 (N.S. 1)} (2021) no. 2, 204--232.

\bibitem{Rev-D} {\sc Sz. Gy. R\'ev\'esz}, Density theorems for the Beurling zeta function. \emph{Mathematika}, {\bf 68} (2022), 1045--1072.

\bibitem{Rev-One} {\sc Sz. Gy. R\'ev\'esz}, Oscillation of the remainder term in the prime number theorem of Beurling, "caused by a given $\zeta$-zero". \emph{Int. Math. Res. Notices}, {\tt https://doi.org/10.1093/imrn/rnac274.} See also at
    {\tt https://arxiv.org/abs/arXiv:2012.09045}.

\bibitem{Rev-NewD} {\sc Sz. Gy. R\'ev\'esz}, The Carlson-type zero-density theorem for the Beurling zeta function.  ArXiv preprint no. arXiv:2209.01689 see at {\tt https://arxiv.org/abs/arXiv:2209.01689}.

\bibitem{Rockafellar} {\sc R. T. Rockafellar}, \emph{Convex Analysis}, Princeton University Press, Princeton, 1970.

\bibitem{Schlage-Puchta} {\sc J.-C. Schlage-Puchta}, Oscillations of the error term in the prime number theorem. \emph{Acta Math. Acad. Sci. Hungar.} {\bf 156} (2018), no. 2, 303--308.

\bibitem{Schlage-PuchtaVindas} {\sc J.-C. Schlage-Puchta} and {\sc J. Vindas}, The prime number theorem for Beurling's generalized numbers. New cases, \emph{Acta Arith.} {\bf 133}(3) (2012) 293--324.

\bibitem{Schmidt}{\sc E. Schmidt}, Über die Anzahl der Primzahlen unter gegebener Grenze, {\em Math. Ann.} {\bf 57} (1903), 195--204.

\bibitem{Stas1} {\sc W. Sta\'s}, \"Uber eine Anwendung der Methode von Tur\'an auf die Theorie des Restgliedes im Primidealsatz, {\it Acta Arith.} {\bf 5} (1959), 179--195.

\bibitem{Stas0} {\sc W. Sta\'s}, \"Uber die Absch\"atzung des Restgliedes im Primzahlsatz, {\it Acta Arith.} {\bf 5} (1959), 427--434.

\bibitem{Stas2} {\sc W. Sta\'s}, \"Uber die Umkehrung eines Satzes von Ingham, {\it Acta Arith.} {\bf 6} (1961), 435--446.

\bibitem{Stas-Wiertelak-1} {\sc W. Sta\'s, K. Wiertelak}, Some estimates in the theory of functions represented by Dirichlet's series, {\it Funct. Approx. Comment. Math.} {\bf 1} (1974), 107--111.

\bibitem{Stas-Wiertelak-2} {\sc W. Sta\'s, K. Wiertelak}, A comparison of certain remainders connected with prime ideals in ideal classes mod $f$, {\it Funct. Approx. Comment. Math.} {\bf 4} (1976), 99--107.

\bibitem{SosTuran}
{\sc S\'os, V. T. and Tur\'an, P.}, On some new theorems in the
theory of diophantine approximations, {\it Acta Math. Acad. Sci.
Hungar.} {\bf 6} (1955), 241--255.


\bibitem{Trud} {\sc T. S. Trudgian}. Updating the error term in the prime number theorem. \emph{Ramanujan J.}, {\bf 39}(2), 225--234 (2016).

\bibitem{Turan} {\sc P. Tur\'an}, {\it On a new method of analysis
and its applications}, Wiley-Interscience, New York, 1984.

\bibitem{Turan1} {\sc P. Tur\'an}, On the remainder term
of the prime-number formula, I, {\it Acta Math. Acad. Sci. Hungar.}
{\bf 1} (1950), 48--63.

\bibitem{Turan2}
{\sc P. Tur\'an}, On the remainder-term of the prime number
formula. II, {\it Acta Math. Acad. Sci. Hungar.} {\bf 1} (1950),
155--160.


\bibitem{V}{\sc la Vall\'ee Poussin, Ch.-J. de }, Sur la
function $\zeta(s)$ de Riemann ... , \emph{Memoaire couronn\'es ...
de Belgique}, {\bf 59} No 1 (1899), 1--74.

\bibitem{Vindas12} {\sc J. Vindas}, Chebyshev estimates for Beurling's generalized prime numbers. I, \emph{J. Number Theory} {\bf 132} (2012) 2371--2376.

\bibitem{Vindas13} {\sc Vindas, J.}, Chebyshev upper estimates for Beurling's generalized prime numbers.
\emph{Bull. Belg. Math. Soc. Simon Stevin} {\bf 20} (2013), no. 1, 175--180.

\bibitem{Zhang19} {\sc Zhang, Wen-Bin},
Exact Wiener-Ikehara theorems. \emph{Acta Arith.} {\bf 187} (2019), no. 4, 357--380.

\bibitem{Zhang15-IJM} {\sc Zhang, Wen-Bin}, Extensions of Beurling's prime number theorem.
\emph{Int. J. Number Theory} {\bf 11} (2015), no. 5, 1589--1616.

\bibitem{Zhang15-MM} {\sc Zhang, Wen-Bin}, A proof of a conjecture of Bateman and Diamond on Beurling generalized primes.
\emph{Monatsh. Math.} 176 (2015), no. 4, 637--656.

\bibitem{Zhang7} {\sc Zhang, Wen-Bin},
Beurling primes with RH and Beurling primes with large oscillation.
\emph{Math. Ann.} {\bf 337} (2007), no. 3, 671--704.

\bibitem{Zhang93} {\sc Wen-Bin Zhang},
Chebyshev type estimates for Beurling generalized prime numbers. II.
\emph{Trans. Amer. Math. Soc.} {\bf 337} (1993), no. 2, 651--675.

\end{thebibliography}
